  \documentclass[sn-mathphys,Numbered]{sn-jnl}

\usepackage{graphicx}%
\usepackage{multirow}%
\usepackage{amsmath,amssymb,amsfonts}%
\usepackage{amsthm}%
\usepackage{mathrsfs}%
\usepackage[title]{appendix}%
\usepackage{xcolor}%
\usepackage{textcomp}%
\usepackage{manyfoot}%
\usepackage{booktabs}%
\usepackage{algorithm}%
\usepackage{algorithmicx}%
\usepackage{algpseudocode}%
\usepackage{listings}%
\usepackage{orcidlink}

\theoremstyle{thmstyleone}%
\newtheorem{theorem}{Theorem}
\newtheorem{proposition}[theorem]{Proposition}%
\newtheorem{lemma}{Lemma}%
\newtheorem{corollary}{Corollary}%

\theoremstyle{thmstyletwo}%
\newtheorem{remark}{Remark}%

\theoremstyle{thmstylethree}%
\newtheorem{definition}{Definition}%

\newcommand{\<}{\langle}
\renewcommand{\>}{\rangle}

\begin{document}

\title[Article Title]{One-sided measure theoretic elliptic operators and applications to SDEs driven by Gaussian white noise with atomic intensity.}


\author[1]{\fnm{Alexandre} \sur{B. Simas} \orcidlink{0000-0003-2562-2829}}\email{alexandre.simas@kaust.edu.sa}
\equalcont{The authors contributed equally and were listed in alphabetical order.}
\author*[1,2]{\fnm{Kelvin} \sur{J. R. Sousa} \orcidlink{0000-0002-9644-9819}}\email{kelvinjhonson.silva@kaust.edu.sa}

\affil[1]{\orgdiv{Statistics Program, Computer, Electrical and Mathematical Science and Engineering Division}, \orgname{King Abdullah University of Science and Technology (KAUST)}, \orgaddress{\city{Thuwal}, \postcode{23955-6900}, \country{Kingdom of Saudi Arabia}}}

\affil[2]{\orgdiv{Department of Mathematics}, \orgname{Universidade Federal da Paraíba (UFPB)}, \orgaddress{\city{João Pessoa}, \postcode{58059-900}, \state{Paraíba}, \country{Brazil}}}


\abstract{We define the operator \(D^+_VD^-_W:=\Delta_{W,V}\) on the one-dimensional torus \(\mathbb{T}\). Here, $W$ and \(V\) are functions inducing (possibly atomic) positive Borel measures on \(\mathbb{T}\), and the derivatives are generalized lateral derivatives. For the first time in this work, the space of test functions \(C^{\infty}_{W,V}(\mathbb{T})\) emerges as the natural regularity space for solutions of the eigenproblem associated with \(\Delta_{W,V}\). Moreover, these spaces are essential for characterizing the energetic space \(H_{W,V}(\mathbb{T})\) as a Sobolev-type space. By observing that the Sobolev-type spaces \(H_{W,V}(\mathbb{T})\) with additional Dirichlet conditions are reproducing kernel Hilbert spaces, we introduce the so-called \(W\)-Brownian bridges as mean-zero Gaussian processes with associated Cameron-Martin spaces derived from these spaces. This framework allows us to introduce \(W\)-Brownian motion as a Feller process with a two-parameter semigroup and c\`adl\`ag sample paths, whose jumps are subordinated to the jumps of \(W\). We establish a deep connection between \(W\)-Brownian motion and these Sobolev-type spaces through their associated Cameron-Martin spaces. Finally, as applications of the developed theory, we demonstrate the existence and uniqueness of related deterministic and stochastic differential equations.}




\maketitle

\section{Introduction}\label{sec1}

Generalized differential operators have emerged as powerful mathematical tools with significant theoretical importance across various domains of mathematics. Our focus centers on the differential operator $D_{W}:=\frac{d}{dW}$, defined with respect to a strictly increasing function $W:\mathbb{R}\to\mathbb{R}$ that is right-continuous (or left-continuous) satisfying a suitable periodic conditions. This operator's significance has been demonstrated in numerous works \cite{frag, pouso, wsimas, trig, uta} and \cite{franco}, particularly in the context of impulsive differential equations, dynamic equations on time scales, and diffusion processes in inhomogeneous media.


The usage of generalized differential operator traces back to the seminal paper \cite{feller}, where the author motivates such a generalized derivatives an alternative to symmetrize one dimensional problems involving the convection-diffusion operator. More precisely the author notes that it is possible to reduce the non self-adjoint operator $\mathfrak{U}=aD^{2}_{s}+bD_{s}+c$, where $a>0, b,c\in\mathbb{R}$, $D_{s}=\frac{d}{ds}$ on $s\in (s_{1}, s_{2})$ to the self-adjoint operator $\mathfrak{U}=-D_{x}D_{W} + c$, where $x\in (x_{1},x_{2})$ and $W$ is a continuous strictly increasing function of $x$. In the same work, the author highlights that $D_{W}$ still makes sense if we just require $W$ to be strictly increasing. Due to the possible discontinuities of $W$ (always of first kind), in general, solutions of equations involving $D_{W}$ possibly have discontinuities of first kind or, informally, have ``jumps'', which might be dense in the interval we are considering $D_{W}$. For that, it is enough, for instance, to require the discontinuity points of $W$ to be dense on its interval of definition. This last observation ensures that regular functions with respect to such derivatives will, most likely, have discontinuity points, and might be very irregular from the standard regularity viewpoint. We also would like to mention that the reduction from the non self-adjoint operator to a self-adjoint operator is very important since it is much more practical to do spectral theory on self-adjoint operators.

From past decades until the present moment, generalized differential equations has been naturally derived to address several interesting problems. For instance, the equation
\begin{equation}\label{apl:2}
\rho_{t}-D_{x}D_{W}\phi(\rho)=0,
\end{equation}
where $D_x$ is the usual differential operator and $D_W$ is a differential operator with respect to a right-continuous and strictly increasing function $W$, arises in \cite{franco}, in the context of hydrodinamic limit of interating particle systems, to model diffusions on permeable membranes.

Another case where such generalized derivative operator plays an important role, arises when considering generalized ordinary differential equations of type 
\begin{equation}\label{apl:1}
x'_{V}(t)=f(t,x(t)),
\end{equation}
which was studied in \cite{pouso}, where $V$ is a càdlàg nondecreasing function. These solutions can be interpreted as solutions of differential equations with impulses or dynamic equations on time scales, depending on the function $V$. This interpretation led to the development of $V$-absolutely continuous functions as a theoretical framework to study (\ref{apl:1}) in the sense of Carathéodory.

Note that the second-order differential operator $\frac{d}{dx}\left(\frac{d}{dW}\right)$ introduced in \cite{frag} has one of its derivatives with respect to $W$, which in this context is a function that may have jumps. Since we are dealing with functions that are differentiable with respect to $W$, typically, we will look for functions that are right-continuous and have left-limits. On the other hand, it is noteworthy that the differential operator $\frac{d}{dW}\left(\frac{d}{dx}\right)$ introduced by Feller in his seminal paper \cite{feller} requires an initial space contained on the space of continuous functions. This is due to the requirement of the existence of the first derivative in the strong sense (and the derivatives being standard bilateral derivatives, not only lateral derivatives). One should note that all the aforementioned operators generalize the usual one-dimensional Laplacian operator $\frac{d}{dx}\left(\frac{d}{dx}\right)$, which is the case when $W(x)=x$. 

In \cite{wsimas} and \cite{simasvalentim2} the authors introduced the $W$-Sobolev spaces by showing that the generalized $W$-derivative can be seen in some sense as weak derivative and obtained some elliptic regularity results related to solutions of generalized linear elliptic equations. However, as will also be discussed in Section \ref{sect4}, the usage of strong derivatives, meaning that the standard limit is taken in the definition of the derivative, instead of just lateral limits, had a serious impact on the regularity study of such operators. Indeed, the elements in the space of regular functions on the $W$-Sobolev spaces, that arise from eigenvectors of the operators, which should be considered smooth, cannot be evaluated pointwisely. This shows the necessity for an adaptation of this theory in such a way that we create a new type of $W$-Sobolev space  (which here will be more general, and we will call $W$-$V$-Sobolev spaces), with good regularity properties so that, one not only has pointwise evaluations, but can also apply the lateral derivatives in the classical sense. Moreover, it also agrees with the $W$-Sobolev spaces in \cite{wsimas} and also with the standard Sobolev spaces, see Remark \ref{sobagree}. 

The first goal of this paper is to consider a general one-sided second order differential operator in such a way that it generalizes at the same time both differential operators we mentioned above, namely $\frac{d}{dW}\left(\frac{d}{dx}\right)$ and $\frac{d}{dx}\left(\frac{d}{dW}\right)$. Indeed, we explain on Remark \ref{difop} how we can re-obtain these differential operators by just imposing regularity conditions on their ranges. Secondly, given the generalized Laplacian $\Delta_{W,V}$, introduced in Definition \ref{laplacianWV}, we define the space $C^{\infty}_{W,V}(\mathbb{T})$ and show that this space contains all solutions of the problems (whenever solutions exist) 
	\begin{equation}\label{eq:1}
		\begin{cases}
			-\Delta_{W,V}\phi = \lambda\phi,\,\lambda\in\mathbb{R}\setminus\{0\};\\
			\phi\in\mathcal{D}_{W,V}(\mathbb{T}),
		\end{cases}
	\end{equation}
where $\mathcal{D}_{W,V}(\mathbb{T})$ is the domain of the Friedrics Extension of the operator $I-D^{+}_{V}D^{-}_{W}$ (cf. Definition \ref{defi5friedri}). To the best of our knowledge, such an explicit characterization can not be found elsewhere in the literature. Moreover, following the ideas of \cite{wsimas} and with the help of the adjoint space $C^{\infty}_{V, W}(\mathbb{T})$ (we would like to ask the reader to pay special attention to the subscripts since they determine in which order the derivatives are taken), we introduce the new $W$-$V$-Sobolev space $H_{W,V}(\mathbb{T})$ and provide some different characterizations for it. The importance of why we should consider the space $C^{\infty}_{W,V}(\mathbb{T})$ comes from the regularity problem associated to the eigenvectors determined by the operator $\mathcal{L}_{W}$ studied in \cite{franco}. Indeed, one expects that each of them (almost surely) inherits, at least, the regularity of the initial space $\mathfrak{D}_{W}$. However, this is not true in general. We refer to the introduction of Section \ref{sect4} to a detailed explanation about the regularity problem.

We, then, apply the theory developed to provide some existence and uniqueness results related to the one-sided second order elliptic operator $L_{W,V}:= -D^{+}_{V}AD_{W}^{-}+\kappa^2$
where $A$ and $\kappa$ satisfies suitable conditions. Further details on the motivation and importance to study $L_{W,V}$ can be found in Section \ref{sect6}.

The remainder of the paper unfolds as follows. In Section \ref{sec:onesideddiff} the one-sided differential operators are introduced and their main properties are derived. Their associated Sobolev-type spaces are introduced in Section \ref{sec:wvsob}. The space of smooth functions, along with their associated aforementioned results are provided in Section \ref{sect4}. Section \ref{sec:wvweak} connects the Sobolev-type spaces to the standard Sobolev spaces by means of weak derivatives, along with results related to approximation by smooth functions. Second-order elliptic equations are studied on Section \ref{sect6}.
The Section \ref{sect7} is dedicated to introduce and study the so called $W$-Brownian bridges, which are mean-zero Gaussian processes with associated Cameron-Martin spaces given by 
$$H_{W,V,\mathcal{D}}(\mathbb{T}) := \{f\in H_{W,V}(\mathbb{T}): f(0) = 0\},$$
where $0\in \mathbb{T}$ any point of $\mathbb{T}$ that is kept fixed (we call this a ``tagged'' point in $\mathbb{T}$). The identification of the Cameron-Martin process of the $W$-Brownian bridge allows us to prove that they are Gaussian Markov random fields.
This further allows us to construct the $W$-Brownian motion on any interval $[0,T]$, $T>0$, and show that it is a Feller process induced by a two-parameter semigroup, which in turn allows us to derive several important probabilistic properties of the $W$-Brownian motion. Finally, we obtain the Cameron-Martin spaces associated to the $W$-Brownian motions in terms of these Sobolev-type spaces, showing a deep connection between the $W$-Brownian motions and the theory developed here. Further, we also show that this Cameron-Martin spaces can be embedded $L_V^2(\mathbb{T})$ for any $V$ satisfying appropriate assumptions. Finally, in Section \ref{sect8}, we establish an existence and uniqueness result on solutions of a Matérn-like elliptic SPDE, namely we consider the equation $L_{W,V}u=B_{W}$, which provides a nice interrelationship between the deterministic and the the non-deterministic tools presented in this work.

\section{The one-sided differential operator}\label{sec:onesideddiff}

In this section we begin by introducing the one-sided differential operators in the classical sense, acting locally on functions. Then, we use it to define a strong second-order differential operator that extends the classical second derivative for a certain class of functions. After that, 
	we show that we can extend this second-order differential operator to a self-adjoint operator which extends the one-dimensional Laplacian. This operator turns out to have a compact resolvent. 
	
Let us introduce some notation. First, we say that a function is c\`adl\`ag (from french, ``continue à droite, limite à gauche'') if the function is right-continuous and has limits from the left. Similarly, we say that a function is c\`agl\`ad, if the function is left-continuous and has limits from the right. 

Let $W, V : \mathbb{R}\to\mathbb{R}$ be strictly increasing functions that are, respectively, c\`adl\`ag and c\`agl\`ad. Further, we assume that they are periodic in the sense that 
	\begin{equation}\label{periodic}
		\forall x \in \mathbb{R},\quad 
		\begin{cases}
			W(x+1)-W(x)=W(1)-W(0);\\
			V(x+1)-V(x)=V(1)-V(0). 
		\end{cases}
	\end{equation}
	
	Without loss of generality we will assume that $W(0-)=W(0)=0$ and $V(0-)= V(0)=0$. 
	Indeed, since they are increasing, they can only have countably many discontinuities points, so if they are discontinuous at zero, we can simply choose another point in which both of them are continuous and perform a translation on the functions to make them continuous at zero. Note that $W$ and $V$ induce finite measures on $\mathbb{T}=\mathbb{R}/\mathbb{Z}$. Additionally, note that we did not impose any restrictions regarding the continuity $W$ or $V$, so throughout the whole work the induced measures are allowed to have atoms. This is a weaker assumption than some of the common assumptions found in literature, e.g., \cite{trig}, \cite{uta},  among others. 
	
	We will denote the $L^2$ space with respect to the measure induced by $V$ by $L^2_V(\mathbb{T})$ and its norm (resp. inner product) by $\|\cdot\|_V$ (resp. $\langle \cdot,\cdot\rangle_V$). Similarly, we will denote the $L^2$ space with the measure induced by $W$ by $L^2_W(\mathbb{T})$ and its norm (resp. inner product) by $\|\cdot\|_W$ (resp. $\langle \cdot,\cdot\rangle_W$).
    
    In what follows, we will define an important class of one-sided differential operators. We begin by providing their definitions in the classical sense, meaning that we see them as pointwise lateral limits, thus they act locally on functions. 
	
	\begin{definition}\label{defgendif}
		We say that a function $f : \mathbb{T}\to\mathbb{R}$ is $W$-left differentiable if  
		$$D^{-}_{W}f(x):=\lim_{h\to 0^{-}}\frac{f(x+h)-f(x)}{W(x+h)-W(x)}$$
		exists for all $x\in \mathbb{T}$. Similarly, we say that a function $g : \mathbb{T}\to\mathbb{R}$ is $V$-right differentiable if the limit 
		$$D^{+}_{V}g(x):=\lim_{h\to 0^{+}}\frac{g(x+h)-g(x)}{V(x+h)-V(x)}$$
		exists for all $x\in \mathbb{T}$.
	\end{definition}

The above definition of generalized lateral derivatives, besides having its natural appeal as a straightforward generalization of Newton's quotient, is also 
similar to the derivative operator when differentiating point processes with respect to a Borel measure (see, for instance, \cite{point}). It is also noteworthy that this definition is very general and will allow a proper regularity study. For instance, \cite{wsimas,simasvalentim2, simasvalentimjde, franco} considered a similar operator when $V(x) = x$, and when the derivative is not a lateral one. The drawback in their approach is that one cannot obtain a pointwise regularity theory. For instance, the eigenvectors of the differential operators considered in \cite{wsimas,simasvalentim2, simasvalentimjde, franco} can only be viewed in the $L^2$-sense, which was the reason the authors in \cite{simasvalentimjde} needed to create an auxiliary space for test functions and also obtain several results exclusively to deal with the lack of regularity. We will discuss this further in Section \ref{sect4}.

A natural question related to Definition \ref{defgendif} is whether there exists a suitable class of functions that are differentiable in this sense. More specifically, we are interested in the functions that are twice differentiable, meaning that they are differentiable with respect to $D_W^-$, and that the resulting function is further differentiable with respect to $D_V^+$. Let us study such a class.
	
	\begin{definition}\label{strongdomain}
		Let $\mathfrak{D}_{W,V}(\mathbb{T})$ be the set of c\`adl\`ag functions $f:\mathbb{T}\to\mathbb{R}$ such that:
		\begin{enumerate}[i)]
			\item $f$ is $W$-left differentiable.
			\item $D^{-}_{W}f$ is a c\`agl\`ad function that is $V$-right differentiable.
			\item $D^{+}_{V}(D^{-}_{W}f)$ is a c\`adl\`ag function.
		\end{enumerate}  
	\end{definition}
	
	Let us prove an auxiliary lemma that will help us in the characterization of the functions in $\mathfrak{D}_{W,V}(\mathbb{T})$. 
	
	\begin{lemma}\label{lm:1}
		If $f : \mathbb{T}\to\mathbb{R}$ is c\`agl\`ad and $D^{+}_{V}f\equiv 0$, then $f$ is a constant function. Similarly, if $g : \mathbb{T}\to\mathbb{R}$ is c\`adl\`ag and $D^{-}_{W}g\equiv 0$, then $g$ is a constant function.
	\end{lemma}
	\begin{proof}
		We will only prove the statement for $f$, since the other is entirely analogous.
	   To this end, let us assume, by contradiction, that $f(a)\neq f(b)$, for some $a,b$, with $a<b$. Let $$\epsilon := \frac{|f(b)-f(a)|}{2\left[V(b)-V(a)\right]},$$
		and $\displaystyle c=\inf \left\{x\in (a,b]; |f(x)-f(a)| > \epsilon \left[V(x)-V(a)\right]\right\}$. Note that 
		$$b\in \left\{x\in (a,b]; |f(x)-f(a)|>\epsilon[V(x)-V(a)]\right\},$$ hence, it is non-empty. Since $f$ if left continuous, we actually have that $c<b$ and $|f(c)-f(a)|\le\epsilon(\left[V(c)-V(a)\right].$ $D^{+}_{V}f\equiv 0$ implies that $(D^{+}_{V}f)(c)=0$. Therefore, there exists $\delta>0$ such that  $$|f(c+h)-f(c)|\le\epsilon\left[V(c+h)-V(c)\right]$$
		whenever $h\in(0,\delta].$
		This, in turn, yields that 
		$$|f(c+h)-f(a)|\le |f(c+h)-f(c)| + |f(c)-f(a)| \le \epsilon \left[V(c+h)-V(a)\right]$$
		whenever $h\in(0,\delta]$. This contradicts the choice of $c$. 
	\end{proof}
	
We are now in a position to characterize the functions in the space $\mathfrak{D}_{W,V}(\mathbb{T})$.
	
	\begin{lemma}\label{lm:2}
		A function $f:\mathbb{T}\to \mathbb{R}$ belongs to $\mathfrak{D}_{W,V}(\mathbb{T})$ if, and only if, there exists a c\`adl\`ag function $g : \mathbb{T}\to\mathbb{R}$ such that
		\begin{equation}\label{c1}
			f(x)=f(0)+W(x)D^{-}_{W}f(0)+\int_{(0,x]}\int_{[0,y)} g(s) dV(s)dW(y),
		\end{equation}
		and
		$$W(1)D^{-}_{W}f(0)+\int_{(0,1]}\int_{[0,y)} g(s) dV(s)dW(y)=0,\;\;\; \int_{[0,1)} g(s) dV(s)=0.$$
	\end{lemma}
	\begin{proof}
		It directly follows from Lemma \ref{lm:1}. The details are left to the reader.
	\end{proof}
	
	\begin{remark}\label{difop}
		Let us consider the special case in which we have $V(x)=x$, and $\mbox{ran}(D^{+}_{V}D^{-}_{W})\subseteq \{\mbox{continuous functions}\}$. In such a case, we have, by \eqref{c1}, that 
	$$D_{W}f(x):=\lim_{h\to 0} \frac{f(x+h)-f(x)}{W(x+h)-W(x)}$$
	is well-defined for all $x$ in $\mathbb{T}$. Further, $D_{W}f$ is continuous and differentiable, with $D_{x}D_{W}f$ also being continuous, i.e., the domain $\mathfrak{D}_{W,V}(\mathbb{T})$ coincides with the domain $\mathfrak{D}_{W}$ of $D_{x}D_{W}$ introduced in \cite{franco}. 
	In the same way, we can work with the operator $D^{-}_WD^{+}_V$ in the domain of c\`agl\`ad functions $f : \mathbb{T}\to\mathbb{R}$ that are: (i) $V$-right differentiable;  (ii) $D^{+}_Vf$ is c\`adl\`ag and $W$-left differentiable; and (iii) $D^{-}_W(D^{+}_Vf)$ is c\`agl\`ad. Therefore if $V(x)$ is continuous and $\mbox{ran}(D^{-}_WD^{+}_V) \subseteq \{\mbox{continuous functions}\}$. Then, then the domain of $D^{-}_WD^{+}_V$ is exactly the domain $\mathfrak{D}(\mathbb{T})$ of the Feller operator $D_{W}D^{-}_{V}$. More precisely, if we define
    $$\overline{W}(x)=W(x-):=\lim_{\delta\to0^+}W(x-\delta),$$
    we have $$D_{\overline{W}}D^{-}_{V}f=D_{W}D^{+}_{V}f.$$
	For more details on the operator $D_{W}D^{-}_{V}$, we refer the reader to \cite{feller, mandl}.
	\end{remark}
	
\begin{remark}\label{orderVW}
		In the remaining of the paper it is important to pay attention at the positions of $W$ and $V$ on the subindex. We will use the subscripts $W,V$ for every structure stritly related to the operator $D^{+}_{V}D^{-}_{W}$. Similarly, whenever we use the subscripts in the order $V,W$, we will be referring to the structure related to the corresponding operator $D^{-}_{W}D^{+}_{V}$. Note that by doing this we will keep changing between c\`adl\`ag and c\`agl\`ad functions. We ask the reader to pay attention to these details as they are subtle.
\end{remark}
	
	\begin{theorem}\label{thm:1}
		The following statements are true
		\begin{enumerate}[a)]
			\item The set $\mathfrak{D}_{W,V}(\mathbb{T})$ is dense in $L^2_{V}(\mathbb{T}).$
            \item (Integration by parts) For $f,g\in\mathfrak{D}_{W,V}$, we have
            		$$\langle D^{+}_{V}D^{-}_{W}f,g\rangle_{V}=-\int_{\mathbb{T}}D^{-}_{W}f(s)D^{-}_{W}g(s) dW(s).$$
			\item The operator $D^{+}_{V}D^{-}_{W}: \mathfrak{D}_{W,V}(\mathbb{T})\subset L^2_{V}(\mathbb{T})\to L^2_{V}(\mathbb{T})$ is symmetric and non-positive.

			\item (Poincaré-Friedrichs Inequality) Let $f$ be a c\`adl\`ag function such that $D^{-}_{W}f$  exists and is a càglàg function. Then,
			\begin{equation}\label{i1}
				\|f\|^{2}_{V}\le W(1)V(1)\|D^{-}_{W}f\|^{2}_{W}+V(1)f_{\mathbb{T}}^{2},
			\end{equation}
			where $f_{\mathbb{T}}=\frac{\int_{\mathbb{T}}f(s)dV(s)}{\int_{\mathbb{T}}1dV(s)}.$
		\end{enumerate}
	\end{theorem}
	\begin{proof}
		Let us prove the first item. Begin by noting that the space of continuous functions is dense in $L^2_{V}(\mathbb{T})$. Therefore, since $\mathbb{T}$ has finite measure, it is enough to show that every continuous function $f : \mathbb{T}\to\mathbb{R}$ can be aproximated, in the uniform norm, by functions in $\mathfrak{D}_{W,V}(\mathbb{T})$. Since $\mathbb{T}$ is compact, given any $\epsilon>0$, there exists $n\in\mathbb{N}$ such that $|x-y|\le 1/n$ implies $|f(x)-f(y)|\le\epsilon$. Now, let us consider the function $g: \mathbb{T}\to \mathbb{R}$
		defined by
		$$g(x)=\sum_{j=0}^{n-1}\dfrac{f\left(\frac{j+1}{n}\right)-f\left(\frac{j}{n}\right)}{W\left(\frac{j+1}{n}\right)-W\left(\frac{j}{n}\right)}\textbf{1}_{\{(j/n,(j+1)/n]\}}(x),$$
		where $\textbf{1}_{A}$ stands for the indicator of the set $A$. Let $G : \mathbb{T}\to\mathbb{R}$ be given by
			$$G(x) = f(0)+\int_{(0,x]}g(y)dW(y).$$
			Therefore, for $\frac{j}{n}<x\le \frac{j+1}{n}$, $0\le j \le n-1$, we have that
		$$G(x)=f\left(\frac{j}{n}\right)+\dfrac{f\left(\frac{j+1}{n}\right)-f\left(\frac{j}{n}\right)}{W\left(\frac{j+1}{n}\right)-W\left(\frac{j}{n}\right)}\left(W(x)-W\left(\frac{j}{n}\right)\right).$$	
		Hence, $G(j/n)=f(j/n)$ and, for $\frac{j}{n}<x< \frac{j+1}{n}$, we have
		$$|G(x)-f(x)|\le|G(x)-f(j/n)|+|f(x)-f(j/n)|.$$
		Now, observe that our choice of $n$ implies $\| G-f \|_{\infty}\le 2 \epsilon$. Note, also, that $g$ is càglàg, and
	    $$
		   \int_{\mathbb{T}}g(s)dW(s)=0.
		$$
		Since $g$ is a c\`agl\`ad function, given $\epsilon>0$ there exists a partition of $\{0=z_{0}<z_{1}<\cdots<z_{k}=1\}$ of $\mathbb{T}$ such that $|g(b)-g(a)|\le \epsilon$, whenever $a,b \in [z_{k-1},z_k)$. Now, let us define a c\`adl\`ag function $p:\mathbb{T}\to\mathbb{R}$ given by 
		$$p(x)=\sum_{i=1}^{k}\dfrac{g(z_{i})-g(z_{i-1})}{V(z_{i})-V(z_{i-1})}\textbf{1}_{\{[z_{i-1},z_{i})\}}(x).$$
		Note that $p$ is c\`adl\`ad and $\int_{\mathbb{T}}p(s)dV(s)=0$. Let us now use this $p$ to define $P:\mathbb{T}\to\mathbb{R}$ by 
		$$P(x)=g(0)+\int_{[0,x)}p(s)dV(s).$$
		By the choice of the partition above, we have that $\|P-g\|_{\infty}\le 2\epsilon.$ Finally, note that for $b=-g(0)-(W(1))^{-1}\int_{(0,1]}dW(\xi)\int_{[0,\xi)}p(\eta)dV(\eta)$, the function
		$$h(x):=f(0)+\int_{(0,x]}dW(\xi)\left(b+g(0)+\int_{[0,\xi)}p(\eta)dV(\eta)\right),$$
		belongs to $\mathfrak{D}_{W,V}$. Indeed, we have that
		\begin{align*}\label{contaconta}
		h(x)&=f(0)+\int_{(0,x]}dW(\xi)\left(g(0)+\int_{[0,\xi)}p(\eta)dV(\eta)\right)+bW(x)\\
		&= f(0)+\int_{(0,x]}P(\xi)dW(\xi) + bW(x).
		\end{align*}
		Now, observe that since $\|P-g\|_\infty\leq 2\epsilon$ and $\|G-f\|_\infty\leq 2\epsilon$, we have
		\begin{align*}
			\left\|h - bW -f\right\|_\infty &\leq \left\|h - bW - G\right\|_\infty+ \|G - f\|_\infty\leq 2\epsilon(1 + W(1)-W(0)).
		\end{align*}
		Now, let us handle the term $bW(x)$. Note that 
		\begin{align*}
		    \nonumber|b W(x)| &= |W(x)|\left| g(0)+(W(1))^{-1}\int_{(0,1]}dW(\xi)\int_{[0,\xi)}p(\eta)dV(\eta) \right|\\ 
		    &\le \left|\int_{(0,1]}\left(g(0)+\int_{[0,\xi)}p(\eta)dV(\eta)\right)dW(\xi) \right| = \left| \int_{(0,1]} P(\xi) dW(\xi)\right|\\
			&\leq \left| \int_{(0,1]} g(\xi) dW(\xi)\right| + 2\epsilon = 2\epsilon.
		\end{align*}
		Since $\int_{\mathbb{T}} g \, dW = 0$, there exists a constant $C > 0$, independent of $n$ and $k$, such that $\|h - f\|_{\infty} \leq C\epsilon$, which proves item 1. in the statement.

		To prove 2., suppose $g$ is a c\`adl\`ag $W$-left differentiable function with $D^{-}_{W}g$ c\`agl\`ad, and that $f$ is a c\`agl\`ad $V$-right differentiable function with $D^{+}_{V}f$ c\`adl\`ag. For any $\epsilon>0$ choose a partition $\{a=z_{1}<z_{2}<\cdots<z_{n+1}=b\}$ of $[a,b]\subset\mathbb{T}$ such that
		\[\left\{ \begin{array}{ll}
			|g(s)-g(t)|\le\epsilon, \mbox{if}\; z_{k}\le s,t<z_{k+1};\\\\
			|f(s)-f(t)|\le\epsilon, \mbox{if} \;z_{k}< s,t\le z_{k+1}.\end{array} \right. \] 
		We, then, have
		\begin{align}\label{eq1}
			\nonumber A_{1}  &= \left|\int_{[a,b)}g(s)D^{+}_{V}f(s) dV - \sum_{k=1}^{n}g(z_{k})[f(z_{k+1})-f(z_{k})]\right| \\
			& = \left|\sum_{k=1}^{n}\int_{[z_{k},z_{k+1})}D^{+}_{V}f(s)[g(s)-g(z_{k})]dV(s)\right|\le \epsilon\|D^{+}_{V}f\|_{\infty}[V(1)].
		\end{align}
		On the other hand
		$$\sum_{k=1}^{n}g(z_{k})[f(z_{k+1})-f(z_{k})]=[f(b)g(b)-f(a)g(a)]-\sum_{k=2}^{n+1}f(z_{k})[g(z_{k})-g(z_{k-1})]$$
		and 
		\begin{align}\label{eq2}
			\nonumber A_{2}  &= \left|\int_{(a,b]}f(s)D^{-}_{W}g(s) dW - \sum_{k=2}^{n+1}f(z_{k})[g(z_{k})-g(z_{k-1})]\right| \\
			& = \left|\sum_{k=1}^{n}\int_{(z_{k-1},z_{k}]}D^{-}_{W}g(s)[f(s)-f(z_{k})]dW(s)\right|\le \epsilon\|D^{-}_{W}g\|_{\infty}[W(1)].
		\end{align}
		Finally, (\ref{eq1}) and (\ref{eq2}) imply
		\begin{equation}
			\nonumber \left| \int_{[a,b)}\!\!\!\!\! g(s)D^{+}_{V}f(s) dV(s)+\int_{(a,b]}\!\!\!\!\!f(s)D^{-}_{W}g(s) dW(s) -[f(b)g(b)-f(a)g(a)]\right|\le \epsilon C,
		\end{equation}
		where $C=\{\|D^{+}_{V}f\|_{\infty}[V(1)]+\|D^{-}_{W}g\|_{\infty}[W(1)]\}$. Since $\epsilon$ is arbitrary, we obtain
		\begin{equation}\label{e3}
			\int_{[a,b)}g(s)D^{+}_{V}f(s) dV(s)=[f(b)g(b)-f(a)g(a)]-\int_{(a,b]}f(s)D^{-}_{W}g(s) dW(s).
		\end{equation}
        Further, using (\ref{e3}), we can easily see that, for $f,g \in \mathfrak{D}_{W,V}(\mathbb{T})$, we have
		\begin{equation}\label{e4}
			\langle D^{+}_{V}D^{-}_{W}f,g\rangle_{V}=\int_{\mathbb{T}}g(s)D^{+}_{V}D^{-}_{W}f(s) dV(s)=-\int_{\mathbb{T}}D^{-}_{W}f(s)D^{-}_{W}g(s) dW(s).
		\end{equation}
        This proves both 2. and 3. in the statement. 
		
		Finally, to prove 4., note that  $$\int_{\mathbb{T}}f^{2}(\xi)dV(\xi)-V(1)f_{\mathbb{T}}^{2}=\dfrac{1}{V(1)^2}\int_{\mathbb{T}}\left(\int_{\mathbb{T}}[f(\xi)-f(\eta)]dV(\eta)\right)^2dV(\xi).$$
		On the right hand side we use that $f(\xi)-f(\eta)=\int_{(\eta,\xi]}D^{-}_{W}f(s)dW(s)$, and then apply H\"older's inequality to obtain (\ref{i1}).
	\end{proof} 

The following consequence of Theorem \ref{thm:1} will be very useful through this work. However, first, we need a definition.
	
	\begin{definition}
		Let $L^2_{W,0}(\mathbb{T})$ be the subspace of $L^2_W(\mathbb{T})$ consisting of functions with mean zero, i.e. 
		$$L^{2}_{W,0}(\mathbb{T}):=\left\{H\in L^{2}_{W}(\mathbb{T}); \int_{\mathbb{T}}HdW=0\right\}.$$ 
	\end{definition}
	\begin{corollary}\label{cr:1}
		The set of $W$-left-derivatives of functions in $\mathfrak{D}_{W,V}$ is dense in $L^{2}_{W,0}(\mathbb{T})$ with respect to the uniform topology. In particular,
		$$\overline{\{D^{-}_{W}f; f\in \mathfrak{D}_{W,V}\}}^{\|\cdot\|_{L^2_{W}(\mathbb{T})}}=L^2_{W,0}(\mathbb{T}).$$
	\end{corollary}
	\begin{proof}
		Fix a function $g\in L^{2}_{W,0}(\mathbb{T})$. By the version of Theorem 1 for $\mathfrak{D}_{V,W}$ (that is, related to the $V,W$-structure, see Remark \ref{orderVW}) there is a sequence $p_{n}\in \mathfrak{D}_{V,W}$ such that $p_{n}\to g$ uniformly. Now, let $$f_{n}(x):=\int_{(0,x]}p_{n}dW-\dfrac{W(x)}{W(1)}\int_{\mathbb{T}}p_{n}dW.$$
		We have that $f_{n}\in \mathfrak{D}_{W,V}$ and $D^{-}_{W}f_n = p_{n}-\frac{1}{W(1)}\int_{\mathbb{T}}p_{n}dW.$  The triangular inequality yields
		\begin{equation}\label{ineqabov}
		    \|D^{-}_{W}f_{n}-g\|_{\infty}\le \left|\frac{1}{W(1)}\int_{\mathbb{T}}p_{n}dW\right|+\|p_{n}-g\|_{\infty}.
		\end{equation}
		Using the uniform convergence of $p_{n}$ and the inequality (\ref{ineqabov}), we obtain that $D^{-}_{W}f_{n}\to g$ uniformly in $\mathbb{T}$. This yields the desired convergence in $L^{2}_{W}(\mathbb{T}).$
	\end{proof}

\begin{remark}
	Notice that Corollary \ref{cr:1} generalizes a claim made in \cite[p.423]{franco}, whose proof was not given. However, here we provide an explicit proof. Our idea for this proof was to exploit the symmetric properties between the theory in $W,V$ and the theory in $V,W$.
\end{remark}

\section{$W$-$V$-Sobolev Spaces}\label{sec:wvsob}
	In this section we will introduce the Sobolev-type spaces associated to the one-sided derivatives introduced in the previous section. Moreover, we will obtain several characterizations connecting these Sobolev-type spaces to the well-known Sobolev spaces. Let us start with a technical definition.

	\begin{definition}\label{sobolev1}
		Let $V,W:\mathbb{R}\to\mathbb{R}$ be increasing functions satisfying the periodic condition \eqref{periodic}. We define the first order $W$-$V$-Sobolev space, denoted by $H_{W,V}(\mathbb{T})$, as the energetic space (in the sense of Zeidler, \cite[Section 5.3]{zeidler}) associated to the operator $(I-D^{+}_{V}D^{-}_{W}): \mathfrak{D}_{W,V}(\mathbb{T})\subset L^2_{V}(\mathbb{T})\to L^2_{V}(\mathbb{T})$. That is, we define the norm $$\|f\|_{1,2}^2 = \langle (I - D^{+}_{V}D^{-}_{W})f,f\rangle_{V}$$
		in $\mathfrak{D}_{W,V}(\mathbb{T})$ and say that $f\in L^2_{V}(\mathbb{T})$ belongs to $H_{W,V}(\mathbb{T})$ if, and only if, the following conditions hold:
		\begin{enumerate}
			\item There exists a sequence $f_n\in \mathfrak{D}_{W,V}(\mathbb{T})$ such that $f_n\to f$ in $L^2_V(\mathbb{T})$;
			\item The sequence $f_n$ is Cauchy with respect to the energetic norm $\|\cdot\|_{1,2}$.
		\end{enumerate}
		A sequence $(f_n)_{n\in\mathbb{N}}$ in $\mathfrak{D}_{W,V}(\mathbb{T})$ satisfying 1 and 2 is called an \emph{admissible sequence}. 
	\end{definition}
	
	\begin{remark}
		Notice that, by condition 2 of Definition \ref{sobolev1}, the norm $\|\cdot\|_{1,2}$ can be uniquely extended to $H_{W,V}(\mathbb{T})$. Therefore, we endow $H_{W,V}(\mathbb{T})$ with this extended norm, which we will also denote by $\|\cdot\|_{1,2}$.
	\end{remark}

	The following theorem characterizes $H_{W,V}(\mathbb{T})$ and is a generalization of the well-known characterization of the first order Sobolev spaces in dimension 1 as the space of absolutely continuous functions with square integrable derivatives.
	
	\begin{theorem}\label{sobchar}
		A function $f\in L^{2}_{V}(\mathbb{T})$ belongs to $H_{W,V}(\mathbb{T})$ if, and only if, there is a function $F\in L^{2}_{W,0}(\mathbb{T})$ and a finite constant c such that
		\begin{equation}\label{c3}
			f(x)=c+\int_{(0,x]}F(s)dW(s),    
		\end{equation}
		$V$-a.e. Furthermore, in such a case the function $F$ is unique, so we denote $D_W^- f := F$.
	\end{theorem}
	\begin{proof}
		Let $f\in H_{W,V}(\mathbb{T})$. By the definition of energetic space and the relation (\ref{e4}), there exists an admissible sequence $(f_{n})_{n\in\mathbb{N}}$ such that $f_{n}\to f$ in $L^{2}_{V}(\mathbb{T})$ and $\left(D^{-}_{W}f_{n}\right)_{n\in\mathbb{N}}$ is Cauchy in $L^{2}_{W}(\mathbb{T})$ norm. Without loss of generality, we can suppose that $D^{-}_{W}f_{n}\to G$ in $L^{2}_{W}(\mathbb{T})$. Define $\displaystyle g(x)=\int_{(0,x]}G(\xi)dW(\xi)$ and note that $g(1)=0$ and
		$$g(y)-g(x)=\int_{(x,y]}G dW=\lim_{n\to \infty}\int_{(x,y]}D^{-}_{W}f_{n}dW=\lim_{n\to \infty}[f_{n}(y)-f_{n}(x)].$$
		Now, we need to prove that, for each fixed $y\in\mathbb{T}$, we have
		\begin{equation}\label{e5} \int_{\mathbb{T}}[f_{n}(y)-f_{n}(x)]dV(x)\to \int_{\mathbb{T}}[g(y)-g(x)]dV(x).
		\end{equation} 
	Indeed, we have that $f_{n}(y)-f_{n}(x)\to g(y)-g(x)$ for each fixed $y$, and by the H\"older's inequality
		$$|f_{n}(y)-f_{n}(x)|^2\le W(1)\int_{\mathbb{T}}\left(D^{-}_{W}f_{n}\right)^{2}dW,$$
		with the term on the right-hand of the above inequality being bounded due to the fact that $\left(D^{-}_{W}f_{n}\right)_{n\in\mathbb{N}}$ is Cauchy. We can now use the dominated convergence theorem to obtain (\ref{e5}). Note that the convergence $f_{n}\to f$ in $L_{V}^{2}(\mathbb{T})$ implies that, $V$-a.e, we have
		\begin{align*}
			V(1)f(y)
                &= \lim_{n\to\infty}\left[V(1)f_{n}(y)-\int_{\mathbb{T}}f_{n}(x)dV(x)\right]+\lim_{n\to\infty}\int_{\mathbb{T}}f_{n}(x)dV(x)\\ &= V(1)g(y)-\int_{\mathbb{T}}g(x)dV+\int_{\mathbb{T}}fdV.
		\end{align*} 
		That is, for $V$-a.e. we have
		$f(y)= c+\int_{(0,y]}G(\xi)dW(\xi),$
		where $c=V(1)^{-1}\left(\int_{\mathbb{T}}fdV-\int_{\mathbb{T}}g(x)dV\right)$. 
		Conversely, if  $F\in L^{2}_{W,0}$ satisfies (\ref{c3}), we have by Corollary \ref{cr:1} that $\left\{D^{-}_{W}g; g\in \mathfrak{D}_{W,V}(\mathbb{T})\right\}$ is dense in the closed subspace $L^{2}_{W,0}(\mathbb{T})$. Now, choose $\left(D^{-}_{W}g_{n}\right)_{n\in\mathbb{N}}$ such that $D^{-}_{W}g_{n}\to F$ in $L^{2}_{W}(\mathbb{T})$, and let
		$f_{n}(x)=c+\int_{(0,x]}D^{-}_{W}g_{n}dW.$
		It follows from the dominated convergence theorem that $f_{n}\to f$ in $L^{2}_{V}(\mathbb{T})$ and $f_{n}\in \mathfrak{D}_{W,V}(\mathbb{T})$ for all $n\in\mathbb{N}$, that is, $g_{n}$ is admissible for $f$. \end{proof}
	\begin{remark}\label{sobagree}
		If we set $V(x)=x$ in the above Lemma, our space agrees with the energetic space considered in \cite{franco}, that is,  $H_{W,x}(\mathbb{T})=H^{1}_{2}(\mathbb{T})$, where $H^{1}_{2}(\mathbb{T})$ is the space introduced in \cite{franco}. 
		It is noteworthy that even though we defined our space in terms of one-sided derivatives, the resulting energetic space for that particular case in which $V(x)=x$ agrees with the space defined in \cite{franco}, where ``two-sided" derivatives were used. Note that the energetic space $H_2^1(\mathbb{T})$ in \cite{franco} coincides with the $W$-Sobolev space in \cite{wsimas} when the dimension is 1. This shows that our approach for the $W$-$V$-Sobolev space is a natural one, as it generalizes the standard $W$-Sobolev spaces, and in particular, our $W$-$V$-Sobolev space also agree with the standard Sobolev space when $V(x)=W(x)=x$.
	\end{remark}

The next result is a $W,V$-version of the well-known Rellich-Kondrachov's theorem. The proof is similar to the proof in \cite[Lemma 3]{franco}. We will present the details here for completeness.

	\begin{theorem}\label{compactemb}
		The embedding $H_{W,V}(\mathbb{T})\hookrightarrow L^{2}_{V}(\mathbb{T})$ is compact.
	\end{theorem}
	\begin{proof}
		Let $u_{n}(x)=c_{n}+\int_{(0,x]}U_{n}dW$ be a bounded sequence in $H_{W,V}(\mathbb{T})$, where $U_{n}\in L^{2}_{V,0}(\mathbb{T})$, $n\in\mathbb{N}$. By the definition of the energetic norm, we have that $\|u_{n}\|^{2}_{1,2}=\|U_{n}\|_{W}^{2}+\|u_{n}\|_{V}\geq\|U_{n}\|_{W}^{2}$. Therefore, $U_{n}$ is bounded in $L^{2}_{W}(\mathbb{T})$. We can then use Cauchy-Swartz's inequality to obtain that $\int_{(0,x]}U_{n}dW$ is bounded in $L^{2}_{V}(\mathbb{T})$. Therefore, 
		$$c_{n}=u_{n}(x)-\int_{(0,x]}U_{n}dW$$
		 is a bounded sequence in $L^{2}_{V}(\mathbb{T})$ or, equivalently (since $\mathbb{T}$ is compact with finite $V$-measure), $c_{n}$ is a bounded sequence in $\mathbb{R}$. Further, since $L^2_W(\mathbb{T})$ is a separable Hilbert space and the sequence $(U_{n})_{n}$ is bounded in $L^{2}_{W}(\mathbb{T})$, there is a subsequence $(U_{n_{k}})_{k\in\mathbb{N}}$ such that $$U_{n_{k}} \rightharpoonup U$$ for some $U \in L^{2}_{W}(\mathbb{T})$. This shows that $c_{n_{k}}\to c$ for some $c\in \mathbb{R}$. Moreover, for all $x\in\mathbb{T}$, we have $\textbf{1}_{(0,x]}\in L^{2}_{W}(\mathbb{T})$, and thus
		$$\lim_{n\to\infty}u_{n_{k}}(x)=\lim_{n\to\infty}\left\{c_{n_{k}}+\int_{(0,x]}U_{n_k}dW\right\}=c+\int_{(0,x]}UdW.$$
		In addition, $\int_{\mathbb{T}}UdW=0$. Finally, $|u_{n_{k}}(x)|^2\le 2c_{n_{k}}^2+2[W(1)]\|U_{n_{k}}\|_{W}^2$. Therefore, we can use the
		 dominated convergence theorem to conclude that $u_{n_{k}}$ converges to $c+\int_{(0,x]}UdW$ in $L^{2}_{V}(\mathbb
		{T})$.
	\end{proof}
	
	Our goal now is to define the $W$-$V$-generalized Laplacian operator. The content of the following definition is the main reason we started by defining $H_{W,V}(\mathbb{T})$ as an energetic space.
	
	\begin{definition}\label{defi5friedri}
		Let $\mathcal{A}: \mathcal{D}_{W,V}(\mathbb{T})\subseteq L^{2}_{V}(\mathbb{T})\to L^{2}_{V}(\mathbb{T})$ be the Friedrichs Extension of the operator $I-D^{+}_{V}D^{-}_{W}$ (we refer the reader to Zeidler \cite[Section 5.5]{zeidler} for further details on Friedrichs extensions). 
		We can characterize the domain $\mathcal{D}_{W,V}(\mathbb{T})$ as the set of functions $f\in L^{2}_{V}(\mathbb{T})$ such that there exists $\mathfrak{f}\in L^{2}_{V,0}(\mathbb{T})$ satisfying 
		\begin{equation}\label{c4}
			f(x)=a+W(x)b+\int_{(0,x]}\int_{[0,y)} \mathfrak{f}(s) dV(s)dW(y),
		\end{equation}
		where $b$ is satisfies the relation
		$$bW(1)+\int_{(0,1]}\int_{[0,y)} \mathfrak{f}(s) dV(s)dW(y)=0.$$
		Moreover,
		\begin{equation}\label{c5}
			-\int_{\mathbb{T}}D^{-}_{W}f D^{-}_{W}g dW= \langle\mathfrak{f},g\rangle_{V}
		\end{equation}
		for all g in $H_{W,V}(\mathbb{T}).$ 
	\end{definition}

\begin{remark}
	Expression \eqref{c5} follows from a straightforward adaptation of the arguments found in \cite[Lemma 4]{franco}.
\end{remark}

	Hence, by \cite[Theorem 5.5.c]{zeidler} and Theorem \ref{compactemb} above, the resolvent of the Friedrichs extension, $\mathcal{A}^{-1}$, is compact. Thus, there exists a complete ortonormal system of functions $(\nu_n)_{n\in\mathbb{N}}$ in $L^{2}_{V}(\mathbb{T})$ such that $\nu_{n}\in H_{W,V}(\mathbb{T})$ for all $n$, and $\nu_{n}$ solves the equation $\mathcal{A}\nu_n=\gamma_n \nu_n,$ for some $\{\gamma_n\}_{n\in\mathbb{N}}\subset\mathbb{R}.$ Furthermore,
	$$1\le \gamma_{1}\le\gamma_{2}\cdots\to\infty$$ 
	as $n\to\infty$. 
	
	\begin{definition}\label{laplacianWV}
		We define the $W$-$V$-Laplacian as 
		$$\Delta_{W,V}=I-\mathcal{A} : \mathcal{D}_{W,V}(\mathbb{T})\subseteq L^{2}_{V}(\mathbb{T})\to L^{2}_{V}(\mathbb{T}).$$
	\end{definition}
	
We have the following integration by parts formula with respect to the $W$-$V$-Laplacian:
	
	\begin{proposition}\label{intbypartsprop}
		(Integration by parts formula) For every $f\in\mathcal{D}_{W,V}(\mathbb{T})$ and $g\in H_{W,V}(\mathbb{T})$, the following expression holds:
		\begin{equation}\label{intbypartsVW}
			\langle -\Delta_{W,V}f,g\rangle_{V} = \int_{\mathbb{T}}D^{-}_{W}fD^{-}_{W}gdW.
		\end{equation}
	\end{proposition}
	\begin{proof}
		We have by (\ref{c4}) that $\Delta_{W,V}f=\mathfrak{f}$. The result thus follows by (\ref{c5}).
	\end{proof} 

Note that in expression \eqref{intbypartsVW}, the spaces change (more precisely, the measures in which the integrals are being taken with respect to change). On the left-hand side of expression \eqref{intbypartsVW}, the integration is being done on $L^2_V(\mathbb{T})$, with respect to the measure induced by $V$, whereas in the right-hand side of \eqref{intbypartsVW}, the integration is being done on $L^2_W(\mathbb{T})$, with respect to the measure induced by $W$.

We conclude this section by mentioning that, given the c\`adl\`ag function $W$, if we take $V=\overline{W}$ (see Remark \ref{difop}), then $W$ and $V$ induce the same measure over $\mathbb{T}$, so similar results, regarding the operator $\frac{d}{dW}\left(\frac{d}{dW}\right)$, may be obtained as a particular case of the present work.

	\section{The Space $C^{\infty}_{W,V}(\mathbb{T})$}\label{sect4}
	
	Our goal in this section is provide a ``nice'' space for functions in which the lateral derivative exist pointwisely. This means, for example, that whenever we have a function on such spaces, say $f$, the Laplacian $\Delta_{W,V}f$ will agree with $D_V^+D_W^-f$, whereas this last expression exists for every point in $\mathbb{T}$. 
	This was a problem with the previous $W$-Sobolev spaces \cite{wsimas,simasvalentim2,simasvalentimjde}, and also when people dealt with such operators such as in \cite{franco, frag} and \cite{farfansimasvalentim}. 
	The problem is that the eigenfunctions for that operator were obtained in $L^2(\mathbb{T})$ and there were no satisfactory regular spaces for such functions. 
	
	With that goal in mind, let us consider the following proposition.
    
\begin{proposition}\label{regregreg}
    If the set of discontinuity points of $W$ is dense in $\mathbb{T}$ and  $f\in \mathfrak{D}_{W}$ (same space from the Remark \ref{difop}) satisfies $\Delta_{W,x}f=D_{x}D_{W}f=\alpha f$, for some $\alpha\neq 0$, then $f\equiv 0$. That is, there are no non-trivial eigenvalues of $\Delta_{W,x}$ inside $\mathfrak{D}_{W}$.
\end{proposition}
\begin{proof}
    We know that if $f$ admits the $W$-derivative, then $f$ is c\`adl\`ag and typically discontinuous, with its discontinuity points being cointained in the set of discontinuity points of $W$. According to the definition of the operator $D_xD_W$, we have that $D_W f$ is differentiable in the classical sense and 
    \begin{equation}\label{soulreg}
        D_x D_W f = \alpha f.
    \end{equation}
    So, (\ref{soulreg}) give us that $D_{x}D_{W}f$ is c\`adl\`ag, since $f$ is c\`adl\`ag. This means that we have a c\`adl\`ag function that is the derivative of some function. We can now use Darboux's theorem from real analysis to conclude that $D_x D_W f$ must be continuous, which in turn, implies that $f$ must be continuous and therefore identically zero, according to the remark given in \cite[Section 2.1]{frag}.
\end{proof}
	Proposition \ref{regregreg} explains why, so far, no pointwise higher-order regularity results have been found for such operators. Indeed, in \cite{simasvalentimjde}, to be able to do point evaluation on ``regular'' functions they needed to define a new space of test functions and, also, prove several results to circumvent the fact that the ``natural'' space spanned by the eigenfunctions was seen as a subspace of $L^2(\mathbb{T})$, and thus, the point evaluation was not well-defined.
	
	We will now define our space of regular functions in which we can apply the operators pointwisely. We will also show that the eigenfunctions of $\Delta_{W,V}$ belong to this space.
	
Start noticing that if $\nu\in L^2_V(\mathbb{T})$ is an eigenvector of $\Delta_{W,V}$ with corresponding eigenvalue $\lambda$ if, and only if, $\nu$ is an eigenvector of $\mathcal{A}$ with eigenvalue $\gamma=1+\lambda$. We, then, have that $\{\nu_{n}\}_{n\in \mathbb{N}}$ is a complete orthonormal system of $L^{2}_{V}(\mathbb{T})$ composed by eigenvectors of $-\Delta_{W,V}$ and each $\nu_{n}$ has associated eigenvalue $\lambda_n = \gamma_n - 1$. Furthermore, $\nu_{n}\in H_{W,V}(\mathbb{T})$ for all $n$. Moreover, we have that
	$$0=\lambda_{0}\le \lambda_{1}\le\lambda_{2}\cdots\to\infty,$$
	as $n\to\infty$, where for each $k\in\mathbb{N} \setminus \{0\}$, $\lambda_k$ satisfies
	
	\begin{equation}\label{autovt}
		-\Delta_{W,V}\nu_{k}=\lambda_{k} \nu_{k}.
	\end{equation} 
	
To state the our next regularity result we need to define the following sets : 
	$$C_{0}(\mathbb{T}):=\left\{f:\mathbb{T}\to\mathbb{R}; f\;\;\mbox{is c\`adl\`ag},\;\int_{\mathbb{T}}fdV=0\right\},$$
	$$C_{1}(\mathbb{T}):=\left\{f:\mathbb{T}\to\mathbb{R}; f\;\;\mbox{is c\`agl\`ad}\int_{\mathbb{T}}fdW=0\right\},$$
	and
	$$C^{n}_{W,V,0}(\mathbb{T}):=\left\{f\in C_{0}(\mathbb{T}); D^{(k)}_{W,V}f\hbox{ exists and }D^{(k)}_{W,V}f\in C_{\sigma(k)}(\mathbb{T}),\,\forall k\le n\right\},$$
	where 
	\[D^{(n)}_{W,V}:=\left\{ \begin{array}{ll}
		\underbrace{D^{-}_{W}D^{+}_{V} \ldots D^{-}_{W}}_{n-factors},\;\mbox{if}\; n\; \mbox{is odd}; \\
		\underbrace{D^{+}_{V}D^{-}_{W} \ldots D^{-}_{W}}_{n-factors},\;\mbox{if}\; n\; \mbox{is even},\end{array} \right. \quad \hbox{and} \quad
	\sigma(n):=\left\{ \begin{array}{ll}
		1,\;\mbox{if}\; n\; \mbox{is odd}; \\
		0,\;\mbox{if}\; n\; \mbox{is even}.\end{array} \right. \] 
	
	We are now in a position to state and prove our main regularity result. Note also that we will also define our first space of regular functions, namely, $C^{\infty}_{W,V,0}(\mathbb{T})$. Later we will increase it by adding (in the sum of spaces sense) the constant functions.
	
	\begin{theorem}[Regularity of eigenfunctions]\label{eigenreg} The solutions of 
		\[\left\{ \begin{array}{ll}
			-\Delta_{W,V}\psi=\lambda \psi,\; \lambda\in\mathbb{R}\setminus\{0\};\\
			\psi\in \mathcal{D}_{W,V}(\mathbb{T})\end{array} \right. \] 
		belong to \begin{equation}\label{cinfinity}
			C^{\infty}_{W,V,0}(\mathbb{T}):=\bigcap_{n\in\mathbb{N}}C^{n}_{W,V,0}(\mathbb{T}).
		\end{equation}
	\end{theorem}
	\begin{proof}
		Begin by observing that we have, by (\ref{c4}), that
		\begin{equation}\label{eq:n}
			\psi(x)=a+bW(x)-\lambda\int_{(0,x]}\int_{[0,y)} \psi(s) dV(s)dW(y)
		\end{equation}
		$V$-a.e with $a$ and $b$ being determined by the relations
		\begin{equation}\label{eq:n+3}
			bW(1)-\lambda\int_{\mathbb{T}}\int_{[0,y)} \psi(s) dV(s)dW(y)=0\,;\,\,\,\int_{\mathbb{T}}\psi dV=0.
		\end{equation}
	This shows that $\psi\in C_0(\mathbb{T})$. This also shows that to conclude the result, it is enough to show that $D_W^-\psi\in C_1(\mathbb{T})$, $D_V^+(D^-_W\psi) \in C_0(\mathbb{T})$ and that $D_V^+(D^-_W\psi) = \psi$, $V$-a.e. Since, then, the remaining orders follow directly by induction.
	Therefore, the next step is to prove that $D_W^-\psi\in C_1(\mathbb{T})$. Indeed, observe that the function ${\int_{[0,y)}\psi(s)dV(s)}$ is c\`agl\`ad and
		\begin{equation}\label{eq:n+2}
			D_{W}^{-}\left(a+bW-\lambda\int_{(0,\cdot]}\int_{[0,\alpha)} \psi(\beta) dV(\beta)dW(\alpha)\right)(y)=b-\lambda\int_{[0,y)}\psi(s)dV(s)
		\end{equation}
		$W$-a.e. Note that (\ref{eq:n+3}) implies the mean zero conditions over (\ref{eq:n+2}). This shows  $D_W^-\psi\in C_1(\mathbb{T})$. Now, observe that  (\ref{eq:n}) implies
		\begin{equation}\label{eq:n+1}
			b-\lambda\!\int_{[0,y)}\hspace{-0.5cm}\psi(s)dV(s)=b-\lambda\int_{[0,y)}\!\!\left[a+\int_{(0,s]}\!\!\left(b-\lambda\!\int_{[0,\alpha)}\hspace{-0.5cm} \psi(\beta) dV(\beta)\right)dW(\alpha)\right]dV(s).
		\end{equation}
		Since $a+bW(x)-\lambda\int_{(0,x]}\int_{[0,y)} \psi(s) dV(s)dW(y)$ is c\`adl\`ag, it follows by (\ref{eq:n+1}) that
		\begin{equation}\label{eq:n+4}
			D^{+}_{V}\left(b-\lambda\int_{[0,\cdot)}\psi(\beta)dV(\beta)\right)(s)=a+\int_{(0,s]}\left(b-\lambda\int_{[0,\alpha)} \psi(\beta) dV(\beta)\right)dW(\alpha).
		\end{equation}
		$V$-a.e. Since (\ref{eq:n+3}) also implies the mean zero conditions over (\ref{eq:n+4}), this shows that $D_V^+(D^-_W\psi) \in C_0(\mathbb{T})$ and that $D_V^+(D^-_W\psi) = \psi$, $V$-a.e. 
	\end{proof}

Note that the notion of regularity provided by Theorem \ref{eigenreg} is a true notion of regularity, in the sense that we can evaluate any eigenfunction as well as their lateral derivatives of any order at every point of the domain. This is a major contribution to the regularity theory related to such operators. 

We can now increase our space by ``adding'' the constant functions.

\begin{definition}
	We define the space of infinitely differentiable $W$-$V$ functions as
	$$C^{\infty}_{W,V}(\mathbb{T}):=\langle 1\rangle \oplus C^{\infty}_{W,V,0}(\mathbb{T}),$$
	where $C^{\infty}_{W,V,0}(\mathbb{T})$ was defined in \eqref{cinfinity}. We also define, in an analogous manner, the spaces $C^{\infty}_{V,W,0}(\mathbb{T})$ and $C^{\infty}_{V,W}(\mathbb{T})$.
\end{definition}

\begin{remark}
	Observe that the spaces $C^\infty_{W,V}(\mathbb{T})$ and $C_{V,W}^\infty(\mathbb{T})$ have an explicit characterization that generalizes, in a standard manner, the space $C^\infty(\mathbb{T})$, which contributes to the notion that these regularity spaces are natural choices. 
\end{remark}

We have the following results regarding the spaces $C^{\infty}_{W,V}(\mathbb{T})$ and $C^{\infty}_{V,W}(\mathbb{T})$:

\begin{proposition}\label{densenesscinfinity}
	The spaces $C^{\infty}_{W,V}(\mathbb{T})$ and $C^{\infty}_{V,W}(\mathbb{T})$ are dense in $L^2_V(\mathbb{T})$ and $L^2_W(\mathbb{T})$, respectively. Furthermore, we have that the set $\left\{D^{-}_{W}g;g\in C^{\infty}_{W,V}(\mathbb{T})\right\}$ is dense in $L^2_{W,0}(\mathbb{T})$ and the set $\left\{D^{+}_{V}g;g\in C^{\infty}_{V,W}(\mathbb{T})\right\}$ is dense in $L^2_{V,0}(\mathbb{T})$.
\end{proposition}
\begin{proof}
    First, observe that, by Theorem \ref{eigenreg} (along with its natural corresponding version to $\Delta_{V,W}$), the eigenvectors of $\Delta_{W,V}$ belong to $C^\infty_{W,V}(\mathbb{T})$ and that the eigenvectors of $\Delta_{V,W}$ belong to $C^\infty_{V,W}(\mathbb{T})$. This gives us the density of $C^\infty_{W,V}(\mathbb{T})$ in $L^2_V(\mathbb{T})$ and of $C^\infty_{V,W}(\mathbb{T})$ in $L^2_W(\mathbb{T})$. In order to proof the the second claim let $A=\left\{D^{-}_{W}g;g\in C^{\infty}_{W,V}(\mathbb{T})\right\}$. Further, observe that from the Lemma \ref{lm:1} the real number $0$ is an eigenvalue of $\Delta_{W,V}$ (same for $\Delta_{V,W}$) which is simple and has associated invariant space given by the space of constant functions over $\mathbb{T}$. Furthermore, every non null eigenvalue of $\Delta_{W,V}$ is an eigenvalue of $\Delta_{V,W}$ and each invariant subspace (that is, each eigenspace) has a fixed dimension. This implies that for each nonzero eigenvalue, we can find the set of the corresponding eigenvectors of $\Delta_{V,W}$ in the set $A$. Now, we can use the fact that the space generated by the set of eigenvectors of $\Delta_{V,W}$ associated to nonzero eigenvalues is dense in $L^{2}_{W,0}(\mathbb{T})$ to conclude that $A$ is dense in $L^2_{W,0}(\mathbb{T})$. In a similar manner, we prove that $\left\{D^{+}_{V}g;g\in C^{\infty}_{V,W}(\mathbb{T})\right\}$ is dense in $L^2_{V,0}(\mathbb{T})$.
 This proves the result.
\end{proof}

\section{$W$-$V$-Sobolev spaces and weak derivatives}\label{sec:wvweak}

Our goal now is to define the notion of $W$ and $V$ lateral weak derivatives and show that the $W$-$V$-Sobolev space can be viewed as a space of functions that have lateral weak derivative with respect to $W$.  Thus, let us define the notion of lateral weak derivative:
	
	\begin{definition}
		A function $f\in L^{2}_{V}(\mathbb{T})$ has a $W$-left \textit{weak} derivative if, and only if, for every $g\in  C^{\infty}_{V,W}(\mathbb{T})$, there exists $F\in L^{2}_{W}(\mathbb{T})$ such that 
		\begin{equation}\label{defWd}
			\int_{\mathbb{T}}fD_{V}^{+}gdV=-\int_{\mathbb{T}}FgdW.
		\end{equation}
	\end{definition}
	\begin{remark}
		By Proposition \ref{densenesscinfinity}, $C^{\infty}_{V,W}(\mathbb{T})$ is dense in $L^2_W(\mathbb{T})$.  This implies the uniqueness of the $W$-left weak derivative defined by \eqref{defWd}. We will denote for a moment the $W$-left weak derivative of $f$ by $\partial^{-}_{W}f.$ 
	\end{remark}
	\begin{remark}
		If there exists an $F$ satisfying (\ref{defWd}) for all $g\in C^{\infty}_{V,W}(\mathbb{T})$, then by taking $g\equiv 1$, we have $\int_{\mathbb{T}}F dW =0$, i.e., $F\in L^{2}_{W,0}(\mathbb{T}).$\end{remark}
	
	Now, define the ($W,V$)-Sobolev space $$\widetilde{H}_{W,V}(\mathbb{T})=\{f\in L^{2}_{V}(\mathbb{T});  \exists\partial _{W}^{-}f\in L^2_{W,0}(\mathbb{T}) \}.$$ 
	
	One can readily prove that $\widetilde{H}_{W,V}(\mathbb{T})$ is a Hilbert space with respect to the energetic norm $\|f\|^{2}_{W,V}=\|f\|^{2}_{V}+\|\partial_{W}^{-}f\|^{2}_{W}$. However, we actually have more:
	
	\begin{theorem}\label{sobequal}
		We have the following equality of sets $\widetilde{H}_{W,V}(\mathbb{T})=H_{W,V}(\mathbb{T})$
	\end{theorem}
	\begin{proof}
		Let $f\in H_{W,V}(\mathbb{T})$ and $(f_{n})_{n}$ is admissible sequence for $f$. Then, by the integration by parts formula (Proposition \ref{intbypartsprop}), we have that
		\begin{equation}\label{x1}
			\int_{\mathbb{T}}f_{n}D^{+}_{V}gdV=-\int_{\mathbb{T}}D^{-}_{W}f_{n}gdW.
		\end{equation}
		for all $g\in C^{\infty}_{V,W}(\mathbb{T})$. So, we can take the limit as $n\to\infty$ to conclude that
		$$\int_{\mathbb{T}}fD^{+}_{V}gdV=-\int_{\mathbb{T}}D^{-}_{W}fgdW,$$
        for all $g\in C^{\infty}_{V,W}(\mathbb{T})$, i.e. $D^{-}_{W}f = \partial^{-}_{W}f$. This implies that ${H}_{W,V}(\mathbb{T})\subseteq \widetilde{H}_{W,V}(\mathbb{T}).$
		For the reverse inclusion, take $f\in \widetilde{H}_{W,V}(\mathbb{T}).$ Choose a sequence $(f_{n})_{n\in\mathbb{N}}$ with $f_{n}\in \mathfrak{D}_{W,V}$, such that $D^{-}_{W}f_{n}\to \partial^{-}_{W}f$ in $L^{2}_{W,0}(\mathbb{T})$. This implies
		\begin{equation}\label{15}
			f_{n}(x)-f_{n}(0)=\int_{(0,x]}D^{-}_{W}f_{n}dW\to\int_{(0,x]}\partial^{-}_{W}fdW.
		\end{equation}
		On the other hand, for all $g\in C^{\infty}_{V,W}(\mathbb{T})$ we have 
		\begin{equation*}\label{16}
			\int_{\mathbb{T}}\left(f_{n}-f_{n}(0)\right)D^{+}_{V}g dV = \int_{\mathbb{T}} f_{n}D^{+}_{V}g dV = -\int_{\mathbb{T}} D^{-}_{W}f_{n}g dW
		\end{equation*}
		as $n\to \infty$. Now, from (\ref{15}) follows that
		$$\int_{\mathbb{T}}\left(\int_{(0,x]}\partial^{-}_{W}fdW\right)D^{+}_{V}gdV=-\int_{\mathbb{T}}\partial^{-}_{W}f g dW=\int_{\mathbb{T}}fD^{+}_{V}g dV,$$
		i.e.
		\begin{equation}\label{3k}
			\int_{\mathbb{T}}\left(c+\int_{(0,x]}\partial^{-}_{W}fdW-f\right)D^{+}_{V}g dV=0.
		\end{equation}
		where $c$ was chosen such that $c+\int_{(0,x]}\partial^{-}_{W}fdW-f\in L^{2}_{V,0}(\mathbb{T})$. This shows that (\ref{3k}) is true for all $g\in C^{\infty}_{V,W}(\mathbb{T})$. Since $\left\{D^{+}_{V}g;g\in C^{\infty}_{V,W}(\mathbb{T})\right\}$ is dense in $L^{2}_{V,0}(\mathbb{T})$ (see Proposition \ref{densenesscinfinity}), we have that
		$f(x)=c+\int_{(0,x]}\partial^{-}_{W}fdW,$
		$V$-a.e. Thus, we have, from Theorem \ref{sobchar}, that $f\in H_{V,W}(\mathbb{T}).$
	\end{proof}

	\begin{remark}
		Note that Theorem \ref{sobequal}, in particular, shows that $\partial^{-}_{W}f=D^{-}_{W}f$.
	\end{remark}

\begin{remark}
	It is important to note that the proof of Theorem \ref{sobequal} corrects a mistake in a previous proof of a similar result \cite[Proposition 2.5]{wsimas}. Indeed, in that proof, Banach-Steinhaus theorem was incorrectly applied. Thus, its consequences are incorrect.
\end{remark}

	\begin{theorem}[Characterizarion of the Sobolev spaces in terms of Fourier coefficients]\label{thm:6}
		The following characterization of $H_{W,V}(\mathbb{T})$ is true 
		$$H_{W,V}(\mathbb{T})=\left\{f\in L^{2}_{V}(\mathbb{T}); f=\alpha_{0}+\sum_{i=1}^{\infty}\alpha_{i}\nu_{i}; \sum_{i=1}^{\infty}\lambda_{i}\alpha_{i}^2<\infty\right\}.$$
	\end{theorem}
	
To prove the above theorem, we need to prove some auxiliary results. First, define $$\mathcal{W}=\left\{ D^{-}_{W}f; f\in H_{W,V}(\mathbb{T}) \right\}\subset L^{2}_{W,0}(\mathbb{T}).$$ 
	
	\begin{lemma}
		$\mathcal{W}$ is a closed subspace of $L^{2}_{W,0}(\mathbb{T})$
	\end{lemma}
	
	\begin{proof}Let us assume, without loss of generality, that we have $D^{-}_{W}f$ with $f\in L^{2}_{V,0}(\mathbb{T})$, otherwise if $f$ belongs to $L^{2}_{V}(\mathbb{T})$, we replace it by $f-\int f dV$, and its $W$-left derivative does not change. Now, let $D^{-}_{W}f_{n}$ be a sequence in $\mathcal{W}$  converging to $g$ in $L^{2}_{W}(\mathbb{T})$. In particular, $D^{-}_{W}f_{n}$ is Cauchy in $L^{2}_{W}(\mathbb{T})$. Thus, by (\ref{i1}), $f_{n}$ is cauchy in $L^{2}_{V}(\mathbb{T})$ and this implies that $f_{n}$ is cauchy in $H_{W,V}(\mathbb{T})$ in the energetic norm. Since $H_{W,V}(\mathbb{T})$ is complete in the energetic norm, there exists $f\in H_{W,V}(\mathbb{T})$ such that $f_{n}\to f$ in energetic norm. That is, $f_{n}\to f $ in  $L^{2}_{V}(\mathbb{T})$, and $D^{-}_{W}f_{n}\to D^{-}_{W}f$ in $L^{2}_{W}(\mathbb{T})$. By the uniqueness of limit, we obtain $g=D^{-}_{W}f$.
	\end{proof}
	
The above lemma tells us that $\mathcal{W}$ is a Hilbert space with respect to the  $L^{2}_{W}(\mathbb{T})$ norm. Let us prove another auxiliary result.
	\begin{lemma}\label{lemaqfaltou}
		The set $\left\{\frac{1}{\sqrt{\lambda_{i}}}D^{-}_{W}\nu_{i}\right\}_{i=1}^{\infty}$, where $\nu_{k}$ is given satisfying (\ref{autovt}), is a complete orthonormal set in $\mathcal{W}$.
		
	\end{lemma}
	
	\begin{proof} First, note that
		
		\begin{align*}
			\left \langle \dfrac{1}{\sqrt{\lambda_{i}}}D^{-}_{W}\nu_{i},\dfrac{1}{\sqrt{\lambda_{i}}}D^{-}_{W}\nu_{i} \right\rangle_{W} &= \dfrac{1}{\sqrt{\lambda_{i}\lambda_{j}}}\int_{\mathbb{T}}D^{-}_{W}\nu_{i}D^{-}_{W}\nu_{j}dW\\
        &= \dfrac{1}{\sqrt{\lambda_{i}\lambda_{j}}}\int_{\mathbb{T}}\nu_{i}D^{+}_{V}D^{-}_{W}\nu_{j}dV = \dfrac{\sqrt{\lambda_{i}}}{\sqrt{\lambda_{j}}}\delta_{i,j},
		\end{align*}
		where $\delta_{i,j}$ stands for the Kronecker's delta. Let us now prove the completeness of the system. If
		$D^{-}_{W}g\perp D^{-}_{W}\nu_{i}$ for all $i=1,2,3,\dots$, then
		$$0=\int_{\mathbb{T}}D^{-}_{W}gD^{-}_{W}\nu_{j}dW = \int_{\mathbb{T}}gD^{+}_{V}D^{-}_{W}\nu_{j}dV \Rightarrow \int_{\mathbb{T}}g\nu_{i}dV=0.$$
		This means that $g$ is constant and, thus, $D^{-}_{W}g=0.$
	\end{proof}
	
	The above lemma helps us in relating the Fourier coefficients of functions in $H_{W,V}(\mathbb{T})$ with the eigenvalues of $\Delta_{W,V}$. Indeed, let $f\in H_{W,V}(\mathbb{T})$. In particular $f\in L^{2}_{V}(\mathbb{T})$. This implies that there are $\{\alpha_{i}\}_{i=0}^{\infty}$ such that
	$$f=\alpha_{0}+\sum_{i=1}^{\infty}\alpha_{i}\nu_{i},\quad \hbox{where}\quad \alpha_{0}=\int_{\mathbb{T}}fdV \quad\hbox{and}\quad\alpha_{i}=\int_{\mathbb{T}}f\nu_{i}dV.$$ 
	We also have $D^{-}_{W}f=\sum_{i=1}^{\infty}\alpha_{i}D^{-}_{W}\nu_{i}.$
	In fact, by the previous lemma
	$D^{-}_{W}f=\sum_{i=1}^{\infty}\beta_{i}\dfrac{D^{-}_{W}\nu_{i}}{\sqrt{\lambda_{i}}},$
	where $\beta_{i}=\dfrac{1}{\sqrt{\lambda_{i}}}\int_{\mathbb{T}}D^{-}_{W}fD^{-}_{W}\nu_{i}dW=\sqrt{\lambda_{i}}\int_{\mathbb{T}}f\nu_{i} dV= \sqrt{\lambda_{i}}\alpha_{i}.$ We are now in a position to prove Theorem \ref{thm:6}.
	\begin{proof} [Proof of Theorem \ref{thm:6} ] Let $f\in H_{W,V}(\mathbb{T})$. By the above remarks, we have that
		$$D^{-}_{W}f = \sum_{i=1}^{\infty} \alpha_{i}D^{-}_{W}\nu_{i}=\sum_{i=1}^{\infty}\sqrt{\lambda_{i}}\alpha_{i}\dfrac{D^{-}_{W}\nu_{i}}{\sqrt{\lambda_{i}}}.$$
		Since $\left\{\frac{1}{\sqrt{\lambda_{i}}}D^{-}_{W}\nu_{i}\right\}_{i=1}^{\infty}$ is a complete orthonormal set we have
		$$\|D^{-}_{W}f\|_{W}^{2}=\sum_{i=1}^{\infty}\lambda_{i}\alpha_{i}^{2}<\infty.$$
		To prove the reverse inclusion, observe that if $f=\alpha_{0}+\sum_{i=1}^{\infty}\alpha_{i}\nu_{i}$ is such that ${\sum_{i=1}^{\infty}\lambda_{i}\alpha_{i}^{2}<\infty,}$ then the sequence $f_{n}=\alpha_{0}+\sum_{i=1}^{n}\alpha_{i}\nu_{i}$ is Cauchy in $H_{W,V}(\mathbb{T})$. This implies that it converges to $f$ in $L^{2}_{V}(\mathbb{T})$. That is, $f_{n}$ is admissible for $f$, because $f_{n}\in \mathfrak{D}_{W,V}$, and by definition
		$$D^{-}_{W}f=\lim_{n\to\infty}D^{-}_{W}f_{n}=\sum_{i=1}^{\infty}\alpha_{i}D^{-}_{W}\nu_{i}.$$
	\end{proof}

	\begin{corollary}\label{coro1}
		We have the following results regarding approximation of functions in the $H_{W,V}(\mathbb{T})$ by smooth functions:
		$$\overline{ {C^{\infty}_{W,V}}(\mathbb{T})}^{\|.\|_{1,2}}=H_{W,V}(\mathbb{T}).$$ 
		Moreover, to verity that $\xi$ the $W$-left weak derivative of $h \in L^{2}_{V}(\mathbb{T})$, we only need to verify that
		$$\int_{\mathbb{T}}\left(h-\int_{\mathbb{T}}hdV\right)D^{+}_{V}g\;dV=\int_{\mathbb{T}}g\xi\;dW$$
		for all $g\in C^{\infty}_{V,W,0}(\mathbb{T}).$
		
	\end{corollary}
	
	We end this section with a characterization of the dual of $H_{W,V}(\mathbb{T})$:
	
	\begin{proposition}\label{dual_sobolev}
		Let $H^{-1}_{W,V}(\mathbb{T})$ be the dual of $H_{W,V}(\mathbb{T})$. We have that $f\in H^{-1}_{W,V}(\mathbb{T})$ if, and only if, there exist $f_0\in L^2_V(\mathbb{T})$ and $f_1\in L^2_W(\mathbb{T})$ such that for every $g\in H_{W,V}(\mathbb{T})$
		$$(f, g) = \int_{\mathbb{T}} f_0(\xi) g(\xi) dV(\xi) + \int_{\mathbb{T}} f_1(\xi) D_W^-g(\xi)dW(\xi).$$
	\end{proposition}
	\begin{proof}
	Since $H_{W,V}(\mathbb{T})$ is a Hilbert Space, we can use Riesz's representation theorem on $f\in H^{-1}_{W,V}(\mathbb{T})$. This means that there is $f_{0}\in H_{W,V}(\mathbb{T})$ such that for every $g\in H_{W,V}(\mathbb{T})$
		$$(f, g) = \int_{\mathbb{T}} f_0(\xi) g(\xi) dV(\xi) + \int_{\mathbb{T}} D^{-}_{W}f_{0}(\xi) D_W^-g(\xi)dW(\xi).$$
		Clearly $D^{-}_{W}f_{0}\in L_{W}^{2}(\mathbb{T}).$ The converse follows immediately from H\"older's inequality.
	\end{proof}

	\section{One-sided second-order elliptic operators}\label{sect6}

Motivated by (\ref{apl:2}) and (\ref{apl:1}), and following \cite{pouso, franco, wsimas, feller}, in this section we study the one-sided second-order elliptic operator 
\begin{equation}\label{defdif}
    L_{W,V}(u):=-D^{+}_{V}AD_{W}^{-}u+\kappa^2u,
\end{equation}
where \( A, \kappa :\mathbb{T}\to\mathbb{R} \) satisfy suitable conditions. The jumps in \( V \) and \( W \), which define \( L_{W,V} \), relate to models incorporating impulses, reflections, or momentum shifts induced by the environment. Consequently, solutions of equations involving \( L_{W,V} \) (e.g., (\ref{elp:1})) inherit these jump properties as they are governed by the discontinuities in \( W \) and \( V \).
    
In what follows, let $A:\mathbb{T} \to \mathbb{R}$ be a positive bounded function, that is, there exists $K\geq 0$ such that for every $x\in\mathbb{T}$, $0< A(x)\leq K$ and bounded away from zero, that is, there exists $A_0>0$ such that for every $x\in \mathbb{T}$, $A(x)\geq A_0$. Let, $\kappa: \mathbb{T}\to \mathbb{R}$ be a bounded function. Sometimes we will suppose that $\kappa$ is bounded away from zero, meaning that there exists some constant $\kappa_0>0$ such that for all $x$, $\kappa(x)\geq\kappa_0$. Finally, consider $B_{W,V}:H_{W,V}(\mathbb{T})\times H_{W,V}(\mathbb{T}) \to \mathbb{R}$, a bilinear and symmetric map, given by
	\begin{equation}\label{bl:1}
		B_{W,V}[u,v]=\int_{\mathbb{T}}AD^{-}_{W}uD^{-}_{W}vdW+\int_{\mathbb{T}}\kappa^{2}uvdV.
	\end{equation}
	
The rest of this section extends classical results for Sobolev spaces (see, for instance \cite{evans}). We start by defining weak solutions in terms of the bilinear forms $B_{W,V}$:
	
	\begin{definition} Let $f\in L^{2}_{V}(\mathbb{T})$. We say that $u\in H_{W,V}(\mathbb{T})$ is a weak solution of the problem
		\begin{equation}\label{elp:1}
			L_{W,V}u=f\,\, \mbox{in}\,\, \mathbb{T}.
		\end{equation}
		if
		$$B[u.v]=(f,v)_V$$
		for all $v\in H_{W,V}(\mathbb{T}).$
	\end{definition}

As is standard in linear non-fractional elliptic partial differential equations, we will apply Lax-Milgram's theorem as a tool to find weak solution for the problem (\ref{elp:1}). Therefore, we first need to establish some energy estimates.
	\begin{proposition}\label{energyest} If $B_{W,V}$ is defined as above, there are $\alpha>0$ and $\beta>0$ such that for all $u,v\in H_{W,V}(\mathbb{T})$
		$$
		\left|B_{W,V}[u,v]\right|\le \alpha \|u\|_{W,V}\|v\|_{W,V}$$
		and for $u\in H_{W,V,0}(\mathbb{T}):= L^{2}_{V,0}(\mathbb{T})\cap H_{W,V}(\mathbb{T})$, we have
		\begin{equation}\label{eng:2}
			B_{W,V}[u,u]\geq \beta\|u\|^{2}_{W,V}.
		\end{equation}

	\end{proposition}
	\begin{proof}
		By (\ref{bl:1}), the assumptions on $A$ and $\kappa$, the triangle inequality and H\"older's inequality, we have that
		$$|B_{W,V}[u,v]|\le L_{0}\|D^{-}_{W}u\|_{W}\|D^{-}_{W}v\|_{W}+L_{1}^2\|u\|_{V}\|v\|_{V},$$
		where $L_{0}:=\sup_{\mathbb{T}}|A|$ and $L_{1}:=\sup_{\mathbb{T}}|\kappa|$. Now, by using $\|D^{-}_{W}u\|_{W}\le \|u\|_{W,V}$ and $\|u\|_{V}\le \|u\|_{W,V}$, we have
		$$B_{W,V}[u,v]\le (L_{0}+L_{1}^2)\|u\|_{W,V}\|v\|_{W,V}.$$
		setting $\alpha=L_{0}+L_{1}^2$, the first part of the statement is proved. For the second assertion note that by (\ref{i1})
		\begin{equation}
			\dfrac{A_{0}}{W(1)V(1)}\|u\|_{V}^{2}\le A_{0}\|D^{-}_{W}u\|^{2}_{W}\le B_{W,V}[u,u].
		\end{equation}
		Therefore,
		\begin{equation}
			\dfrac{A_{0}}{2}\min\left\{\dfrac{1}{W(1)V(1)},1\right\}\|u\|_{W,V}^2\le B_{W,V}[u,u].
		\end{equation}
	\end{proof}
	
	\begin{remark}\label{eng:1}
		Observe that if we consider the new operator $$L_{W,V,\lambda}u:=L_{W,V}u+\lambda u,$$
		then the energy estimates still hold for  
		$$B_{W,V,\lambda}[u,v]:=B_{W,V}[u,v]+\lambda(u,v)_{V}.$$
		In this case, the second assertion is true for all $u\in H_{W,V}(\mathbb{T})$ and $\lambda>0$. Moreover $\alpha$ and $\beta$ depends on $\lambda.$ Finally, if in addition $\kappa$ is bounded away from zero we get (\ref{eng:2}) for $\lambda>-\kappa_{0}$.
	\end{remark}

	We can now use standard arguments to obtain existance and uniqueness of solutions to \eqref{elp:1}.
	
	\begin{proposition}\label{prop:40}
		Given $f\in L^2_{V}(\mathbb{T}),$ there exists a unique $u\in H_{W,V,0}(\mathbb{T})$ that is a weak solution of 
		\begin{equation}\label{slc:1}
			L_{W,V}u=f\,\,\mbox{in}\,\,\mathbb{T}.
		\end{equation}
		Moreover, by the Remark \ref{eng:1}, for each $\lambda>0$, there is a unique $u\in H_{W,V}(\mathbb{T})$ such that
		\begin{equation}\label{atv:1}
			L_{W,V}u+\lambda u = f,\,\,\mbox{in}\,\,\mathbb{T}.
		\end{equation}
		If, additionally, $\kappa$ is bounded away from zero, then there exists $u\in H_{W,V}(\mathbb{T})$ that is a weak solution of (\ref{slc:1}) and (\ref{atv:1}) can be weakly solved for $\lambda>-\kappa_{0}.$ Furthermore, the solutions of (\ref{slc:1}) and (\ref{atv:1}) satisfies the following bounds 
		\begin{equation}\label{ine:10}
			\|u\|_{W,V}\le C\|f\|_{V}    
		\end{equation}
		for some constant $C>0$ independent of $f$ and
		$$\|u\|_{V}\le \lambda^{-1}\|f\|_{V}$$
		for $\lambda>0.$ For $\kappa$ bounded away from zero we have that 
		\begin{equation}\label{atv:2}
			\|u\|_{V}\le (\kappa_{0}+\lambda)^{-1}\|f\|_{V}.
		\end{equation}
		for $\lambda>-\kappa_{0}.$
	\end{proposition}
	\begin{proof}
		By Proposition \ref{energyest} and Remark \ref{eng:1}, the  existence and uniqueness follows from Lax-Milgram's theorem. For the bounds note that
		$$\beta\|u\|_{W,V}^2\le B_{W,V}[u,u] = (f,u)_{V} \le \|f\|_{V}\|u\|_{V} \le \|f\|_{V}\|u\|_{W,V}$$
		that is, $\|u\|_{W,V}\le C\|f\|_{V},$ for $C=\beta^{-1}.$
		Analogously,
		$$\lambda\|u\|^2_{V}\le A_{0}\|D^{-}_{W}u\|^2_{V}+\lambda\|u\|^2_{V}\le B_{W,V,\lambda}[u,u] = (f,u) \le \|f\|_{V}\|u\|_{V}.$$
		Finally, (\ref{atv:2}) can be obtained similarly.
	\end{proof}
	
	We also have the following consequence to Fredholm's alternative:

	\begin{proposition}
		Precisely one of the following assertions are true: 
		\begin{enumerate}
			\item For $f\in L^2_{V}(\mathbb{T})$, there exists a unique solution of 
			\begin{equation}\label{fa:1}
				L_{W,V}u=f,
			\end{equation} 
			or else
			\item There is a weak solution of $u\not\equiv 0$ of the homogeneous problem
			$$L_{W,V}u=0.$$
		\end{enumerate}
		Moreover, if 2. is true, we have $\dim \ker L_{W,V}<\infty$ and (\ref{fa:1}) has a weak solution if and only if 
		$$\int_{\mathbb{T}}vfdV=0$$
		for every $v\in\ker L_{W,V}.$
	\end{proposition}
	\begin{proof}
		Fix $\lambda>0$. From Proposition \ref{prop:40}, given $g\in L^{2}_{V}(\mathbb{T})$ there exists a unique $u\in H_{W,V}(\mathbb{T})$ such that 
		$$B_{W,V,\lambda}[u,v]=(g,v)_{V}.$$
		Therefore, by existence and uniqueness, we can invert the operator on $g$ to obtain $L_{W,V,\lambda}^{-1}g:=u.$ Now note that $u\in H_{W,V}(\mathbb{T})$ satisfies \eqref{fa:1} if, and only if, 		for all $v\in H_{W,V}(\mathbb{T})$, we have
		$$B_{W,V,\lambda}[u,v]=(\lambda u+f,v)_{V}$$
        or equivalently, if, and only if, $L_{W,V,\lambda}^{-1}(\lambda u+f)=u$
		which holds if, and only if, $u-Ku=h,$
		where $h:=L_{W,V,\lambda}^{-1}f$ and $Ku:=\lambda L^{-1}_{W,V,\lambda}u$. We can now use bounds obtained on Proposition \ref{prop:40} and Theorem \ref{compactemb} to conclude that the operator $K:L_{V}^{2}(\mathbb{T})\to H_{W,V}(\mathbb{T})\subset L^{2}_{V}(\mathbb{T})$ is compact and self-adjoint, this last is due to the symmetry of $B_{W,V}$. The result now follows from Fredholm's alternative.
	\end{proof}
	
The previous result shows that $\ker L_{W,V}$ plays a key role in the study of weak solutions of (\ref{fa:1}). More explicitly, we have
 $$ker L_{W,V}=\left\{w\in H_{W,V}(\mathbb{T}); \forall v\in H_{W,V}(\mathbb{T}), B_{W,V}[w,v] = 0\right\}.$$
	In particular, for $v=w\in \ker L_{W,V}$ we have that
	$$\dfrac{A_{0}}{W(1)V(1)}\left\|w-\dfrac{1}{V(1)}\int_{\mathbb{T}}wdV\right\|^2_{V}\le A_{0}\|D^{-}_{W}w\|^2\le B_{W,V}[w,w]=0.$$
    So
    $$w\equiv \dfrac{1}{V(1)}\int_{\mathbb{T}}wdV.$$ 
	Thus, if there exists $w\in \ker L_{W,V}$ such that $\int_{\mathbb{T}}wdV\big/V(1)\not\equiv 0$, then $\kappa\equiv 0.$ In this case, there exists $u\in H_{W,V}(\mathbb{T})$ that is a weak solution of 
    $D^{+}_{V}AD^{-}_{W}u=f$
	if and only if 
	$\int_{\mathbb{T}}fdV=0.$
	Furthermore, the solution is unique in $H_{W,V}(\mathbb{T})$ up to a constant. On the other hand, if $\ker L_{W,V}=0$, then
	$L_{W,V}$ define a bijection between $H_{W,V}(\mathbb{T})$ and $L^{2}_{V}(\mathbb{T})$. This happens, for instance, when $\kappa$ is bounded away from zero.

	\begin{proposition}Let $\kappa$ bounded away from zero. The following assertions are true
		\begin{enumerate}
			\item The eigenvalues of $L_{W,V}$ are real, countable, and we enumerate them, according their multiplicity, in such a way that
			\begin{equation}\label{eig:1}
				0<\lambda_{1}\le \lambda_2 \le \ldots \to \infty.
			\end{equation}
			\item There exists an orthonormal basis $\{\phi_{k}\}_{i\in\mathbb{N}}$ of $L^{2}_{V}(\mathbb{T})$ where $\phi_{k}\in H_{W,V}(\mathbb{T})$ is a eigenvector associated to $\lambda_{k}$, i.e.
			$$L_{W,V}\phi_{k}=\lambda_{k}\phi_{k}\,\,in\,\,\mathbb{T}.$$
		\end{enumerate}
	\end{proposition}
	\begin{proof} 
		We will prove 1. and 2. simultaneously. First of all, the linear operator $L_{W,V}:H_{W,V}(\mathbb{T})\to L^{2}_{V}(\mathbb{T})$ defined by
		$$L_{W,V}u=f \Leftrightarrow B_{W,V}[u,v]=(f,v), \forall v\in H_{W,V}(\mathbb{T})$$
		is well defined and is a bijection. Indeed, $L^{-1}_{W,V}: L^{2}_{V}(\mathbb{T})\to L^{2}_{V}(\mathbb{T})$ is linear operator, which is compact and symmetric. Moreover, by (\ref{ine:10}), if $L_{W,V}u=f$, there exists $C>0$ such that
		$$\|u\|_{V}\le\|u\|_{W,V}\le C \|f\|_{V} \Rightarrow \|L^{-1}_{W,V}f\|_{V}\le \|f\|_{V}.$$
		As we know, the immersion of $H_{W,V}(\mathbb{T})$ in $L^{2}(\mathbb{T})$ is compact, therefore $L^{-1}_{W,V}$ is compact. The symmetry easily follows from the symmetry of $B_{W,V}$. Clearly $0$ is not an eigenvalue of $L_{W,V}$ and, if $\mu\neq 0$ and $u\neq 0$ we have $$L_{W,V}u=\mu u \Leftrightarrow L^{-1}_{W,V}u=\frac{1}{\mu}u.$$ Therefore, since the eigenvectors of a symmetric operator belongs to $\mathbb{R}$, the eigenvectors of $L_{W,V}$ belongs to $\mathbb{R}.$ If $u$ and $\mu$ are taken as above, then the inequality
		$$\kappa_{0}\|u\|^{2}\le B_{W,V}[u,u]=\mu\|u\|^{2}$$
		implies that $0<\kappa_{0}\le\mu$. Hence, we can apply the spectral theorem for self-adjoint and compact operators on a separable Hilbert space to the operator $L_{W,V}^{-1}$. Therefore, there exists a complete orthonormal set $\{\phi_{k}\}_{k\in\mathbb{N}}$ of $L^{2}_{V}(\mathbb{T})$ constituted by eigenvectors of $L^{-1}_{W,V}$ where, each $\phi_{k}$ is associated to $\frac{1}{\lambda_{k}}$ with \begin{equation*}
			\frac{1}{\lambda_{1}}\geq \frac{1}{\lambda_{2}}\geq \ldots \geq \frac{1}{\lambda_{k}} \to 0.
		\end{equation*}
		This last fact is equivalent to (\ref{eig:1}).
	\end{proof}

\begin{remark} It noteworthy that the
operator $L_{W,V}$ was introduced in a very natural manner. Further, the existence of weak solutions as well as the results related to the eigenproblem
associated to $L_{W,V}$ followed very easily by using a modern machinery for that. Therefore, so far our approach do not look exhaustive
nor artificial. However, in \cite{feller} Feller did not want to consider this problem directly in that form. One possible reason for that may be due for the fact almost at the same time that Feller introduced its generalized second order operator, the theory of Sobolev spaces was gaining space as a suitable tool to study elliptic problems.
\end{remark}

\section{$W$-Brownian bridge and $W$-Brownian motion}\label{sect7}

In this section we will consider the Sobolev space $H_{W,V}(\mathbb{T})$ augmented with a Dirichlet-type condition, which we will show is a reproducing kernel Hilbert space, to define a generalization of the Brownian bridge as the mean zero Gaussian process with covariance function given by the reproducing kernel of that space. Then, we will use the $W$-Brownian bridge to define a generalization of the well-known Brownian motion, which we will denote by $B_W(t)$, such that has the following \emph{simultaneous} interesting features:

\begin{enumerate}
	\item The finite-dimensional distributions of $B_W$ are Gaussian;
	\item The sample paths of $B_W$ are c\`adl\`ag and may have jumps.
\end{enumerate}

We have a well-known family of processes that generalize the standard Brownian motion and satisfies condition 1, indeed, the fractional Brownian motion is such an example. However, the sample paths of the fractional Brownian motion are continuous, in fact, $\gamma$-H\"older continuous, for $\gamma<H$, where $H$ is the Hurst parameter of the fractional Brownian motion. We also have a well-known family of stochastic processes that
generalize the Brownian motion and that satisfies 2, namely the L\'evy processes. However, by Lévy–Khintchine characterization, its finite-dimensional distributions are Gaussian if, and only if, it is the Brownian motion. Therefore, it is unusual to have a process that generalizes the Brownian motion having Gaussian finite-dimensional distributions and that has jumps. In what follows, to keep notation simple, we will assume in this Section that $W(0) = 0$. 

We start by augmenting the space $H_{W,V}(\mathbb{T})$ with the Dirichlet-type condition. First, take any point in $\mathbb{T}$ and ``tag" it by identifying it as $zero$. Then, consider the standard identification between the torus and the interval $[0,1)$. Second, define
\begin{equation*}\label{eq:dirichletSob}
	H_{W,V,\mathcal{D}}(\mathbb{T}) = \{f\in H_{W,V}(\mathbb{T}): f(0) = 0\}
\end{equation*}
and endow it with the ``Dirichlet'' inner product:
\begin{equation*}\label{eq:dirichletinnerproduct}
\<f,g\>_{W,V,\mathcal{D}} := \int_{\mathbb{T}} D_W^- f D_W^- g dW.
\end{equation*}
It is immediate that $(H_{W,V,\mathcal{D}}(\mathbb{T}), \<\cdot,\cdot\>_{W,V,\mathcal{D}})$ is a Hilbert space. Further, for $s,t \in [0,1)$, let
$$\varrho_{W,0}(t,s) = W(t\land s) - \frac{W(t)W(s)}{W(1)}.$$
We have, from a straightforward computation, that for $t\in\mathbb{T}$, ${\varrho_{W,0}(\cdot, t)\in H_{W,V,\mathcal{D}}(\mathbb{T})}$ and that, for every $f\in H_{W,V,\mathcal{D}}(\mathbb{T})$ and every $t\in [0,1)$,
$$\<f,\rho_{W,0}(\cdot,t)\>_{W,V,\mathcal{D}} = f(t).$$
That is, $\varrho_{W,0}(\cdot,\cdot)$ is a reproducing kernel for $(H_{W,V,\mathcal{D}}(\mathbb{T}), \<\cdot,\cdot\>_{W,V,\mathcal{D}})$. In particular, the Sobolev-type space $(H_{W,V,\mathcal{D}}, \<\cdot,\cdot\>_{W,V,\mathcal{D}})$ is a reproducing kernel Hilbert space, e.g., \cite[Theorem 10.2]{wendland2004scattered}.

We are now in a position to define the $W$-Brownian bridge:

\begin{definition}\label{def:wbrownianbridge}
	Fix some point in $\mathbb{T}$ and ``tag'' it as zero. We define the $W$-Brownian bridge on $[0,1)\cong\mathbb{T}$, denoted by $B_{W,0}(\cdot)$, as the centered Gaussian process (that is Gaussian process with expected value zero) with covariance function given by $\varrho_{W,0}(\cdot,\cdot)$. 
\end{definition}

\begin{remark}
    By \cite[Theorem 10.4]{wendland2004scattered}, since $\varrho_{W,0}(\cdot,\cdot)$ is a reproducing kernel, it is positive semi-definite and, therefore, is a valid covariance function.
\end{remark}

We first show that $B_{W,0}(\cdot)$ indeed satisfies a bridge property. To this end, we will now work on the interval $[0,1]$. 
Let $C_{\mathcal{D}}([0,1])$ denote the space of càdlàg functions defined on $[0,1]$ that vanishes at $x=0$ and $x=1$ and define
$$H_{W,\mathcal{D}}([0,1]) =\left\{f\in C_{\mathcal{D}}([0,1]): \exists F\in L^{2}_{W,0}([0,1]); f(x)=\int_{0}^{x} F(s) dW(s)\right\}$$
endowed with the inner product 
\begin{equation*}\label{eq:dir_inner_prod}
    \<f,g\>_{W,\mathcal{D}} := \int_{[0,1]} D_W^- f D_W^- g dW.
\end{equation*}
It is clear that if we take any strictly increasing function $V:\mathbb{R}\to\mathbb{R}$ satisfying \eqref{periodic}, 
then, we have a natural identification 
\begin{equation}\label{eq:identification_dir_spaces}
    H_{W,\mathcal{D}}([0,1]) \cong H_{W,V,\mathcal{D}}(\mathbb{T}).
\end{equation}
This means from the left-hand side of \eqref{eq:identification_dir_spaces} that the underlying space is intrinsic relative to $V$ and, from the right-hand side, the underlying space might be embedded on any $L^{2}_{V}(\mathbb{T})$. In particular, we also point out that, this identification implies that $H_{W,\mathcal{D}}([0,1])$ is also a reproducing kernel Hilbert space with the same kernel $\varrho_{W,0}(\cdot,\cdot)$.

\begin{remark}\label{rem:wbbinterval}
	By the above considerations, it is clear that we can consider $\varrho_{W,0}:[0,1]\times[0,1]\to\mathbb{R}$ and define the $W$-Brownian bridge $B_{W,0}(\cdot)$ as the centered Gaussian process with covariance function $\varrho_{W,0}(\cdot,\cdot)$ on the interval $[0,1]$.
\end{remark}

We now prove the bridge property:

\begin{proposition}\label{prp:bridgeproperty}
	The $W$-Brownian bridge satisfies the following bridge process: $B_{W,0}(0) = B_{W,0}(1) = 0$ almost surely. 
\end{proposition}
\begin{proof}
By Definition \ref{def:wbrownianbridge} and Remark \ref{rem:wbbinterval}, we have that
$$\textrm{Var}(B_{W,0}(0)) = \varrho_{W,0}(0,0) = 0 = \varrho_{W,0}(1,1) = \textrm{Var}(B_{W,0}(1)),$$
since we are assuming that $W(0) = 0$. This implies that $B_{W,0}(0)$ and $B_{W,0}(1)$ are, almost surely, constant equal to their expected values. Since $B_{W,0}$ has mean zero, the result follows.
\end{proof}

\begin{remark}
	There are two reasons we call the process $B_{W,0}(\cdot)$ in Definition \ref{def:wbrownianbridge} as the $W$-Brownian bridge. The first one is Proposition \ref{prp:bridgeproperty}, whereas the second one is 
	that if we take $W(x)=x$, we recover the covariance function of the standard Brownian bridge on $[0,1]$, e.g., \cite[Definition 40.2]{rogers2000diffusions}.
\end{remark}

The first consequence of this construction is that we directly have the Cameron-Martin space associated to $B_{W,0}(\cdot)$:

\begin{proposition}\label{prp:cmwbrownianbridge}
	Let $B_{W,0}(\cdot)$ be the $W$-Brownian bridge on $[0,1]$. Then, the Cameron--Martin space associated to $B_{W,0}(\cdot)$ is $(H_{W,\mathcal{D}}, \<\cdot,\cdot\>_{W,\mathcal{D}})$.
\end{proposition}
\begin{proof}
It follows from \cite[Theorem 8.15 and Corollary 8.16]{janson_gaussian} that the Cameron-Martin space associated to the centered Gaussian process $B_{W,0}(\cdot)$ is the reproducing kernel Hilbert space, whose reproducing kernel is given by its covariance function, $\varrho_{W,0}(\cdot,\cdot)$, that is, the Cameron-Martin space is $(H_{W,\mathcal{D}}, \<\cdot,\cdot\>_{W,\mathcal{D}})$.
\end{proof}

Observe that Proposition \ref{prp:cmwbrownianbridge} establishes a deep and strong connection between the $W$-Brownian bridge, and the Sobolev-type spaces developed in this paper. The Cameron-Martin
spaces appear naturally on Malliavin calculus as the set of ``directions'' in which one can differentiate (with respect to the Malliavin derivative). Recently Bolin and Kirchner \cite{bolinkirchner} showed that the Cameron-Martin
spaces play a key role in find the best linear predictor with the ``wrong'' covariance structure. More generally, the Cameron-Martin space uniquely determines the Gaussian measure (see \cite[Theorem 2.9]{daprato}) on Banach spaces. Furthermore, the Cameron-Martin space is, in a sense, independent of the Banach space in which the Gaussian measure is defined (see \cite[Proposition 2.10]{daprato}). This means that one can specify a Gaussian measure by specifing the Cameron-Martin space. This means that we can take any strictly increasing $V:\mathbb{R}\to\mathbb{R}$ satisfying \eqref{periodic} and, by computing the Cameron-Martin space of $B_W$, we are showing that the distribution of $W$-Brownian motion in $L^2_V(\mathbb{T})$ is the unique Gaussian measure defined on the Borel sets of $L^2_V(\mathbb{T})$ associated to the Sobolev space $H_{W,V,\mathcal{D}}(\mathbb{T})$, which is identified with $H_{W,\mathcal{D}}([0,1])$, which means that the $W$-Brownian motion takes values in $L^2_V(\mathbb{T})$. Finally, several applications of Cameron-Martin spaces to statistics and probability can be found, e.g., in \cite{berlinet2009reproducing,gine2021mathematical}.

\begin{remark}\label{remark_L2_v}
    The fact that we can take any $V$ so that the $W$-Brownian motion takes values on $L^2_V(\mathbb{T})$ will be important in the next section, where it will play a key role in obtaining existence and uniqueness of solutions of second-order linear stochastic differential equations.
\end{remark}

We will now provide another important application of the identification of the Cameron-Martin space of $B_{W,0}(\cdot)$ in Proposition \ref{prp:cmwbrownianbridge}. More precisely, we will use Proposition \ref{prp:cmwbrownianbridge} together with the results in \cite{kunsch1979gaussian} to show that the $W$-Brownian motion is a Gaussian Markov random field.

This is a non-trivial application of Proposition \ref{prp:cmwbrownianbridge}, since the standard way of proving that the standard Brownian bridge satisfies a Markov property requires stochastic calculus, which we have not yet developed for such generalized differential operators. Indeed, the standard way is by showing that the Brownian bridge is an Itô process, and then use that Itô processes are Markov processes, e.g., \cite[Exercise 5.4 an
d Theorem 7.1.2]{oksendal2003stochastic}. 

We start with the following result:

\begin{proposition}\label{prp:markovfield}
	The $W$-Brownian bridge on $[0,1]$ is a Gaussian Markov random field in the sense of \cite{kunsch1979gaussian}.
\end{proposition}
\begin{proof}
	Since the space $(H_{W,\mathcal{D}}([0,1]), \<\cdot,\cdot\>_{W,\mathcal{D}})$ is local, it is, thus, an immediate consequence of Proposition \ref{prp:cmwbrownianbridge} and \cite[Theorem 5.1]{kunsch1979gaussian}.
\end{proof}

For a Borel set $S\subset [0,1]$, let
$$\mathcal{F}_+^{B_{W,0}}(S) = \bigcap_{\epsilon>0} \sigma(B_{W,0}(s): s\in S_\epsilon),$$
where $S_\epsilon = \{t\in [0,1]: d(t,S)<\epsilon\}$ is the $\epsilon$-neighborhood of $S$. Proposition \ref{prp:markovfield} implies, in particular, that for every $s\in [0,1]$ and every integrable (with respect to the probability measure) $U$ that is $\mathcal{F}_+^{B_{W,0}}([t,1])$-measurable, we have
\begin{equation}\label{eq:markovidentity}
	E(U|\mathcal{F}_+^{B_{W,0}}([0,s])) = E(U|\mathcal{F}_+^{B_{W,0}}(\{s\})).
\end{equation}

In particular, by taking $U = f(B_{W,0}(t))$ in \eqref{eq:markovidentity}, where $f$ is any bounded measurable function and $0\leq s<t\leq 1$, we obtain that 
$$	E(f(B_{W,0}(t))|\mathcal{F}_+^{B_{W,0}}([0,s])) = E(f(B_{W,0}(t))|\mathcal{F}_+^{B_{W,0}}(\{s\})).$$

We are now in a position to use the $W$-Brownian bridge to define the $W$-Brownian motion on $[0,1]$:

\begin{definition}\label{def:wbrownianmotion}
	Let $B_{W,0}(\cdot)$ be a $W$-Brownian bridge on $[0,1]$ and let $B_1\sim N(0,W(1))$ be an Gaussian random variable that is independent of $\sigma(B_{W,0}(t): t\in [0,1])$. We define the $W$-Brownian motion (in law) as the process
	$$B_W(t) = B_{W,0}(t) + \frac{W(t)}{W(1)} B_1.$$
\end{definition}

We have some immediate consequences:

\begin{proposition}\label{prp:wbrowniancovfunc}
	Let $B_W(\cdot)$ be a $W$-Brownian motion (in law) on $[0,1]$. Its covariance function is given by 
	$$\varrho_W(t,s) = W(t\land s), \quad t,s\in [0,1].$$
	Furthermore, $B_W(\cdot)$ is a $W$-Brownian motion (in law) if, and only if, 
	$$S(t) = B_W(t) - \frac{W(t)}{W(1)} B_W(1)$$
	is a $W$-Brownian bridge. In particular, the $W$-Brownian bridge can be obtained from $W$-Brownian motion (in law) $B_W(\cdot)$ by conditioning $B_W(\cdot)$ on $B_W(1) = 0$.
\end{proposition}

\begin{proof}
	Both claims follow directly from Definitions \ref{def:wbrownianbridge} and \ref{def:wbrownianmotion}. The last statement follows from \cite[Remark 9.10]{janson_gaussian}.
 \end{proof}






The expression for the covariance function $\varrho_W(\cdot,\cdot)$ allows us to also obtain a characterization of the $W$-Brownian motion that is very reminiscent of the classical one:
\begin{proposition}\label{W-brownian-char}
	 $B_W(t)$ is a $W$-Brownian motion (in law) if, and only if, it satisfies the following conditions:
	\begin{enumerate}
		\item $B_W(0) = 0$ almost surely;
		\item If $t>s$, then $B_W(t)-B_W(s)$ is independent of $\sigma(B_W(u); u\leq s)$;
		\item If $t>s$, then $B_W(t)-B_W(s)$ has $N(0,W(t)-W(s))$ distribution.
	\end{enumerate}
\end{proposition}
\begin{proof}
Condition 1. follows from the assumption that $W(0)=0$ and that $B_W(\cdot)$ has mean zero. Conditions 2. and 3. are immediate consequences of the fact that $B_W(\cdot)$ is a Gaussian process whose covariance function is $\varrho_W(t,s) = W(t\land s), t,s\in [0,1]$. Conversely, if conditions 1-3 hold, then, for $t>s$,
\begin{align*}
	\textrm{Cov}(B_W(t),B_W(s)) &= E(B_W(t)B_W(s)) \\
	&= E((B_W(t)-B_W(s))B_W(s)) + E(B_W(s)B_W(s))\\
	&= E(B_W(s)B_W(s))=W(s),
\end{align*}
since $B_W(s) = B_W(s) - B_W(0)\sim N(0,W(s))$.
\end{proof}

Figure \ref{fig:1}, below, contains a realization of a $W$-Brownian bridge and of a $W$-Brownian motion for a function $W(x)$ such that for $0\leq x < 0.25$, $W(x) = x/2$, for $0.25\leq x < 0.6$, $W(x) = 1+2x$ and for $0.6\leq x\leq 1$, $W(x) = 2+\exp(2x)$.

\begin{figure}[h]
    \centering
    \includegraphics[scale=0.30]{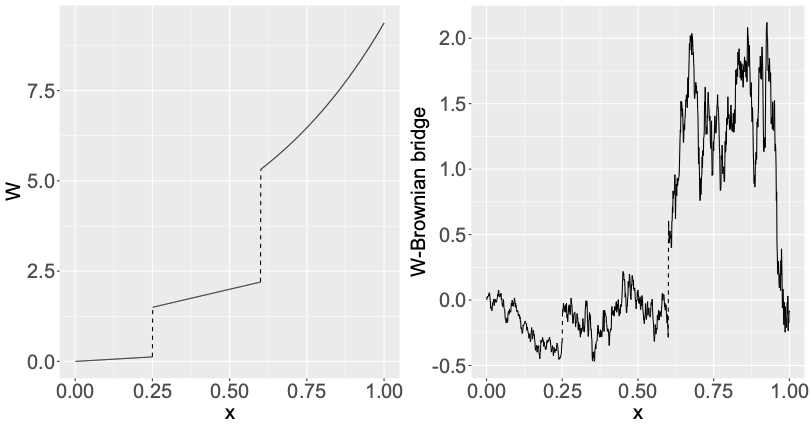}
        \includegraphics[scale=0.30]{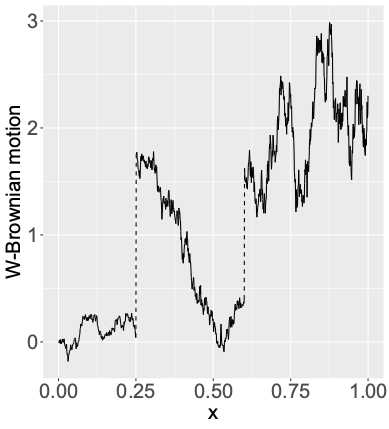}
    \caption{An example of a realization of a $W$-Brownian bridge (middle) and of a $W$-Brownian motion (right) for a specific function $W$ (left).}
    \label{fig:1}
\end{figure}

We now define the $W$-Brownian motion on $[0,T], T>0$, by additionally requiring that the sample paths are c\`adl\`ag:

\begin{definition}\label{W-brownian}
	We say that $B_W(t)$ is a $W$-Brownian motion if it satisfies the following conditions:
	\begin{enumerate}
		\item $B_W(0) = 0$ almost surely;
		\item If $t>s$, then $B_W(t)-B_W(s)$ is independent of $\sigma(B_W(u); u\leq s)$;
		\item If $t>s$, then $B_W(t)-B_W(s)$ has $N(0,W(t)-W(s))$ distribution;
		\item $B_W(s)$ has c\`adl\`ag sample paths.
	\end{enumerate}
\end{definition}

\begin{remark}
	Notice that if $t^\ast$ is a discontinuity point of $W$, then there exists some $\varepsilon>0$ such that for every $h>0$,
	$$P(|B_W(t^\ast-h) - B_W(t^\ast)|>0) \geq \varepsilon.$$
	Indeed, it is enough to take any $\varepsilon>0$ such that $P(|N(0,\Delta W(t^\ast)|>0) > \varepsilon$, where $\Delta W(t^\ast) = W(t^\ast) - W(t^\ast-)$.
\end{remark}

\begin{remark}
	The previous remark then implies that $B_W$ is not continuous in probability so, in particular, $B_W$ is not an additive process (see, for instance, \cite{sato} for a definition of additive processes). Indeed, it is easy to see that the $W$-Brownian motion is an additive process if, and only if, $W$ is continuous.
	Furthermore, it is also straightforward that the $W$-Brownian motion is a Lévy process if, and only if, $W(x) = ax$ for some $a>0$.
\end{remark}

Our next goal is to show that $B_W(\cdot)$ admits, almost surely, a c\`adl\`ag modification. This will be obtained as a particular case of the fact (that we are going to show) that $B_W(\cdot)$ is actually a Feller process.

We, thus, start by showing that $B_W(\cdot)$ is a Markov process. Observe that, contrary to the $W$-Brownian bridge, it is actually trivial to directly check the Markov property:

\begin{proposition}\label{prp:markovWB}
	Let $B_W(\cdot)$ be a $W$-Brownian motion in law and let $\mathcal{F}_t:=\sigma(B_W(s):s\leq t)$ be its natural filtration. Then, $B_W(\cdot)$ is a martingale and a Markov process with respect to $(\mathcal{F}_t)_{t\in [0,1]}$.
\end{proposition}
\begin{proof}
	By Condition 3 of Definition \ref{W-brownian}, $E(B_W(t) - B_W(s)) = 0$, and by Condition 2 of \ref{W-brownian}, we have that, for $t>s$,
	\begin{eqnarray*}
		E(B_W(t)|\sigma(B_W(u), u\leq s)) &=&  E(B_W(t)- B_W(s)|\sigma(B_W(u), u\leq s)) + B_W(s)\\
		&=& E(B_W(t)- B_W(s)) + B_W(s)= B_W(s).
	\end{eqnarray*}
This shows that $B_W(\cdot)$ is a martingale. Let us now show that $B_W(\cdot)$ is a Markov process. To this end, let $A$ be a Borel set in $[0,1]$, $0\leq s<t\leq 1$ and observe that since $B_W(t)-B_W(s)$ is independent of $\mathcal{F}_s$, it follows that, e.g. from \cite[Problem 5.9]{karatzas2012brownian}, that
\begin{align*}
	E(\textbf{1}_{B_W(t) \in A}|\mathcal{F}_s) &= E(\textbf{1}_{B_W(t) - B_W(s) + B_W(s) \in A}|\mathcal{F}_s)\\
	&= E(\textbf{1}_{B_W(t) - B_W(s) + B_W(s) \in A}|B_W(s)) = E(\textbf{1}_{B_W(t) \in A}|B_W(s)).
\end{align*}
This proves the Markov property.
\end{proof}

Consider the following definition that has an important role in the preceding discussion. 

\begin{definition}\label{def:markov_semi}
    A two parameter Markov semigroup consists in a family $\{T_{s,t}\}_{s,t\in\mathbb{R}}$, of bounded linear operators from $L^{\infty}(\mathbb{R})$ to $L^{\infty}(\mathbb{R})$ if the following conditions are satisfied for every $0\leq s\leq u\leq t$:

    \begin{enumerate}
        \item $T_{s,s}=I$, for every $t\in \mathbb{R};$
        \item $T_{s,t} 1 = 1$;
        \item if $f\in L^\infty(\mathbb{R})$ and $f\geq 0$, then $T_{s,t} f \geq 0$;
        \item $T_{s,t} = T_{s,u}T_{u,t}$ for every $s\le u\le t$.
    \end{enumerate}
    
\end{definition}

\begin{corollary}\label{cor:markovsemigroup}
	Let $B_W(\cdot)$ be a $W$-Brownian motion (in law) and let, for $x\in \mathbb{R}$ and a measurable function $f$,
	$$E^{s,x}(f(B_W(t))) := E(f(B_W(t+s) + x)).$$
	Let, also, for each $s,t\in [0,1]$, with $s\leq t$, $T_{s,t}: L^\infty(\mathbb{R})\to L^\infty(\mathbb{R})$ be given by
	$$(T_{s,t}f)(x) := E^{s,x}(f(B_W(t) - B_W(s))).$$
	Then, the family $(T_{s,t})_{s\leq t\in [0,1]}$ is a two-parameter Markov semigroup on $[0,1]$.
\end{corollary}
\begin{proof}
We need to check the conditions in Definition \ref{def:markov_semi}. However, all of them are straightforward, except for the semigroup property 
$T_{s,t} = T_{s,u}T_{u,t},$
with $s\leq u\leq t$, which follows directly from Proposition \ref{prp:markovWB}.
\end{proof}

Recall the following definition

\begin{definition}\label{two_par_feller}
    A two parameter Markov semigroup (in the sense of Definition \ref{def:markov_semi} is said to be a Feller semigroup if the following conditions are satisfied for all $0\leq s\leq t$:
    \begin{enumerate}
        \item $T_{s,t}$ is a bounded operator from $C_0(\mathbb{R})$ into $C_0(\mathbb{R})$;
        \item for every $f\in C_0(\mathbb{R})$, $\|T_{s,t} f\|_{\infty} \leq \|f\|_{\infty}$;
        \item for all $0\leq v\leq w$, we have 
        $$\lim_{\substack{(s,t)\to (v,w)\\ s\leq t}} \left\|T_{v,w}-T_{s,t} \right\|=0.$$
    \end{enumerate}
\end{definition}

We are now in a position to prove that $B_W(\cdot)$ is a Feller process in the sense of Definition \ref{two_par_feller}.

\begin{theorem}\label{thm:WBFeller}
	Let $B_W(\cdot)$ be a $W$-Brownian motion (in law) on $[0,1]$, then $B_W(\cdot)$ is a Feller process.
\end{theorem}
\begin{proof}
Observe that since $B_W(t)\sim N(0,W(t))$, so that $B_W(t)+x\sim N(x,W(t))$, it follows that for $f\in L^\infty(\mathbb{R})$ (and, in particular, for $f\in C_0(\mathbb{R})$),
$$(T_{s,t}f)(x) = \int_{\mathbb{R}} p_{s,t}(x,y) f(y) dy,$$
where $p_{s,t}(x,y) = (2\pi (W(t)-W(s))^{-1/2} \exp\left\{ -(y-x)^2/(2(W(t)-W(s)))\right\}.$
Now, in view of \cite[Theorem 3.2]{bottcher2014feller} and \cite[Proposition 2.4.1]{khoshnevisan2002multiparameter}, it is enough to observe that for fixed $s>0$ and any $\delta>0$,
$$\sup_{x\in\mathbb{R}}\int_{y\in\mathbb{R}: |x-y|\geq \delta} p_{s,t}(x,y) dy = \sqrt{\frac{2}{\pi}} \int_{\delta (W(t)-W(s))^{-1/2}}^\infty e^{-z^2/2} dz \stackrel{t\to s+}{\longrightarrow} 0,$$
since $W$ is right-continuous. 
\end{proof}

\begin{remark}
	The explicit expression for $T_{s,t}f$ in the proof of Theorem \ref{thm:WBFeller} allows us to show that $u_{f,s}(t,x) := (T_{s,t}f)(x)$ solves the following Kolmogorov-type equation:
	$$\begin{cases}
		\partial_{W,t}^+ u_{f,s}(t,x) = \frac{\partial^2}{\partial x^2} u_{f,s}(t,x),\\
	u_{f,s}(s,x) = f(x),
	\end{cases}$$
where $\partial_{W,t}^+$ stands for the $W$-right-derivative operator with respect to ``time'' $t$, and $\frac{\partial^2}{\partial x^2}$ is the standard second-order partial derivative with respect to $x$.
\end{remark}

Observe that Theorem \ref{thm:WBFeller} directly gives us several good properties for the $W$-Brownian motion in law. First, that it admits, almost surely, a c\`adl\`ag modification. This shows, in particular, that the $W$-Brownian motion given in Definition \ref{W-brownian} exists. Furthermore, let $\mathcal{F}_{s,t} := \sigma(B_W(t): t\geq s)$ (note that $\mathcal{F}_t = \mathcal{F}_{0,t})$, and also let
$$\overline{\mathcal{F}}_{s,t}^\ast := \bigcap_{u>t} \overline{\mathcal{F}}_{s,u}, \quad t\geq s,$$
be the complete augmented filtration induced by $B_W(\cdot)$, where $\overline{\mathcal{F}}_{s,t}$ is the $P$-completion of $\mathcal{F}_{s,t}$, which consists of adding the $P$-null sets to $\overline{\mathcal{F}}_{s,t}$. 
The direct consequences of Theorem \ref{thm:WBFeller} are summarized in the following corollary:

\begin{corollary}\label{cor:conseFeller}
	Let $B_W(\cdot)$ be a $W$-Brownian motion (in law) on $[0,1]$. Then,
	\begin{enumerate}
		\item $B_W(\cdot)$ admits a c\`adl\`ag modification that is itself a Feller process with the same transition function as $B_W(\cdot)$;
		\item The $W$-Brownian motion given in Definition \ref{W-brownian} exists, in particular the $W$-Brownian bridge has a modification with c\`adl\`ag sample paths;
		\item Fix any $s \in [0,1)$ and let $B_W(\cdot)$ be a $W$-Brownian motion. Then, $B_W(\cdot) - B_W(s)$ is also a Feller process with respect to its complete augmented filtration $(\overline{\mathcal{F}}_{s,t}^\ast)_{t\in [0,1]}$;
		\item Let $B_W(\cdot)$ be a $W$-Brownian motion (as in Definition \ref{W-brownian}), then $B_W(\cdot)$ is a strong Markov process;
		\item $B_W(\cdot)$ satisfies a Blumenthal's 0-1 law: for any $s \in [0,1]$ and $A\in \overline{\mathcal{F}}^\ast_{s,s}$, we have either $P(A) =0$ or $P(A)=1$.
	\end{enumerate}
\end{corollary}

\begin{proof}
Item 1 follows from straightforward adaptations of \cite[Theorem 3.2]{bottcher2014feller} and \cite[Theorem 4.1.1]{khoshnevisan2002multiparameter} and \cite[Theorem (2.7), p.91]{revuz2013continuous}, item 2 is an immediate consequence of 1 and Proposition \ref{prp:wbrowniancovfunc}. Item 3 follows from straightforward adaptations of \cite[Theorem 3.2]{bottcher2014feller} and  \cite[Theorem 4.1.2]{khoshnevisan2002multiparameter}, whereas item 4 follows from adaptations of \cite[Theorem 3.2]{bottcher2014feller} and \cite[Theorem 4.2.2]{khoshnevisan2002multiparameter}. Finally, item 5 is an immediate consequence of item 3. 
\end{proof}

If $W$ has finitely many discontinuity points and $W$ is H\"older continuous on each continuous subinterval of the form $[d_i, d_{i+1})$, where $\{d_i; i=1,\ldots, N\}$ is the set of discontinuity points, where $N\in\mathbb{N}$, then the set of discontinuity points of the sample paths will be contained in the set of discontinuity points of $W$:
\begin{proposition}\label{kolmchen}
	If the set of discontinuity points of $W$ is finite, say $0\leq d_1 < d_2 <\ldots < d_N$, for some $n\in\mathbb{N}$, and if we let $I_1 = [0,d_1), I_2 = [d_1,d_2), \ldots, I_{N} = [d_{N-1},d_N), I_{N+1} = [d_N,1)$, and for each interval $I_k$, there exists some $\gamma_k\in (0,1]$, such that the restriction of $W$ to $I_k$ is $\gamma_k$-H\"older continuous, then $B_W$ has a modification whose sample paths are continuous on each interval $I_k, k=1,\ldots, N+1$.
\end{proposition}
\begin{proof}
	In this case, we can directly apply Kolmogorov-Chentsov's continuity theorem to the restriction of $B_W$ to each interval $I_k$ to obtain that the restriction of $B_W$ to each $I_k$ has a modification, say $B_{W,k}$, which has continuous sample paths, for $k=1,\ldots, N+1$. Notice that from the very construction, the modifications belong to the same probability space. Therefore, we can simply define the modification $\widetilde{B}_W$ on $[0,1]$ as $\widetilde{B}_W(s) = B_{W,k}(s)$ is $s\in I_k$, $k=1,\ldots, N+1$, and since the intervals $I_k$ are disjoint, there are no overlaps.
\end{proof}

\begin{remark}
	It is important to notice that if we only assume $W$ to have finitely many discontinuity points, and impose no assumptions regarding H\"older continuity, we would only have that the restriction of $B_W$ to each continuity subinterval is continuous in probability, which implies that the restriction of $B_W$ to each subinterval is an additive process in law, and thus admits a version which is right-continuous. However it is not strong enough to ensure the existence of a modification with continuous sample paths. Thus, we would not be able to conclude that $B_W$ would be continuous in each interval such that $W$ is continuous. 
\end{remark}

Notice that, almost surely, the sample paths of $B_W(\cdot)$ are bounded on $[0,1]$, since it is c\`adl\`ag on $[0,1]$. Therefore,
it is an $L_W^2(\mathbb{T})$-process (see Remark \ref{remark_L2_v}). Furthermore, $B_W(\cdot)$ has orthogonal increments and is $L^2$-right-continuous, that is $$\lim_{t\to s+} \mathbb{E}((B_W(t)-B_W(s))^2) =0.$$ Therefore, the definition of the stochastic integral of any function $f\in L^2_W(\mathbb{T})$ with respect to $B_W(\cdot)$ follows from the standard theory on $L^2$-processes (see \cite{ash}). Namely, the stochastic integral is linear, 
\begin{equation}
	\int_{[0,1)} \textbf{1}_{(a,b]} dB_W = B_W(b) - B_W(a),
	\label{intsimplefunctions}
\end{equation}
and, for any $f\in L^2_W(\mathbb{T})$, we have the following isometry:
\begin{equation}
	\mathbb{E}\left[ \left( \int_{[0,1)} f dB_W\right)^2\right] = \int_{[0,1)} f^2 dW.
	\label{isometry}
\end{equation}
Now, observe that, by \eqref{intsimplefunctions}, the integral of each simple function follows a normal distribution. By \eqref{isometry}, the stochastic integral with respect to $B_W$ is an $L^2$-limit of Gaussian random variables, and thus, it is a Gaussian random variable (see \cite[Theorem 1.4.2 ]{ash}). 

Since, for each simple function $\varphi$, we have $\mathbb{E}\left( \int_{[0,1)} \varphi dB_W \right) =0$, it follows from \eqref{isometry} that the stochastic integral with respect to any deterministic function $f\in L^2_W(\mathbb{T})$ is zero. Furthermore, since the stochastic integral with respect to deterministic functions is a mean zero Gaussian random variable, it follows from \eqref{isometry}, again, that for any $f\in L^2_W(\mathbb{T})$,
\begin{equation}
	\int_{[0,1)} f dB_W \sim N\left(0, \int_{[0,1)} f^2 dW\right).
	\label{diststochint}
\end{equation}

We can now show that the integration by parts formula can be applied for any function in $H_{W,V}(\mathbb{T})\cap L_W^2(\mathbb{T})$ and with the derivative, being the weak derivative $D_W^-$:

\begin{proposition}\label{pathintegral}
	Let $B_W(\cdot)$ be a $W$-Brownian motion. For any function $g\in H_{W,V}(\mathbb{T})\cap L^2_W(\mathbb{T})$ we have the following integration by parts formula:
	\begin{equation}\label{stochintpath}
		\int_{(0,t]} g dB_W = B_W(t)g(t) - \int_{(0,t]} B_W(s-) D_W^-{g}(s) dW(s).
	\end{equation}
\end{proposition}
\begin{proof}
	Let $g\in H_{W,V}(\mathbb{T})$. From (an immediate adaptation of) Theorem \ref{sobchar} and stochastic Fubini's theorem (see \cite[Theorem 64, p.210]{protter}), we have
	\begin{align*}
		\int_{(0,t]} g dB_W &= \int_{(0,t]} \left( g(0) + \int_{(0,s]} D_W^-g(u)dW(u)\right) dB_W(s)\\
		&= g(0) B_W(t) + B_W(t) (g(t) - g(0)) - \int_{(0,t]} D_W^-g(u) B_W(u-)dW(u)\\
		&= B_W(t)g(t) - \int_{(0,t]} D_W^-g(u) B_W(u-)dW(u).
	\end{align*}
\end{proof}

Notice that since $B_W(\cdot)$ has c\`adl\`ag sample paths, $B_W(\cdot)$ is almost surely bounded on $[0,1]$, so the right-hand side of \eqref{stochintpath} is well-defined for any function $g\in H_{W,V}(\mathbb{T})$ (even if it is not in $L_W^2(\mathbb{T})$). Proposition \ref{pathintegral} thus motivates the following definition:

\begin{definition}\label{def:stochintHWV}
	Let $g\in H_{W,V}(\mathbb{T})$, the stochastic integral of $g$ with respect to $B_W(\cdot)$ is defined as
	$$\int_{(0,t]} g dB_W := B_W(t)g(t) - \int_{(0,t]} D_W^-g(u) B_W(u-)dW(u).$$
\end{definition}

\section{Applications to stochastic differential equations}\label{sect8}

We will now apply the theory developed in the previous sections of this paper to solve a class of stochastic partial differential
equations that generalize a non-fractional case of the well-known Matérn equations on a domain with periodic boundary conditions. 

For this section, we will work on $H_{W,V,\mathcal{D}}(\mathbb{T})$ with a tagged zero and, furthermore, recall that we assumed that $W$ is continuous at $x=0$. Further, for simplicity, we will assume that $W(0)=0$. Therefore, from Theorem \ref{sobchar}, we have that for any function $g\in H_{W,V}(\mathbb{T})$,
$$\lim_{x\to 1-} g(x) = g(0) = \lim_{x\to 0+} g(x).$$
This means that if $g\in H_{W,V,\mathcal{D}}(\mathbb{T}),$
$$\lim_{x\to 1-} g(x) = 0 = \lim_{x\to 0+} g(x).$$
Next, we identify $\mathbb{T} = [0,1)$, so that, from Definition \ref{W-brownian}, we can define the $W$-Brownian motion $B_W(\cdot)$ on $\mathbb{T}$ as its restriction to $[0,1)$. Therefore, it follows from Definition \ref{def:stochintHWV} that under these assumptions, for any $g\in H_{W,V,\mathcal{D}}(\mathbb{T}),$ we
have
	$$\int_{[0,1)} g(s) dB_W(s) = - \int_{[0,1)} B_W(s-) D_W^-{g}(s) dW(s).$$
	So that
	\begin{eqnarray}
		\left|\int_{[0,1)} g(s) dB_W(s)\right| &\leq& \int_{[0,1)} |B_W(s) D_W^-{g}(s)| dW(s)\nonumber \leq \|B_W\|_{L^2_W([0,1))} \|D_W^{-}g\|_{L^2_W(\mathbb{T})}\nonumber\\
		&\leq& \left(\sup_{t\in [0,1)}\|B_W(t)\|\right) W(1) \|g\|_{H_{W,V}(\mathbb{T})}.\label{functionalwhite}
	\end{eqnarray}

	So, the stochastic integral defines, almost surely, a bounded linear functional on $H_{W,V,\mathcal{D}}(\mathbb{T})$.
	We can now define the pathwise $W$-gaussian white noise on the Cameron-Martin space $H_{W,V,\mathcal{D}}(\mathbb{T})$.

\begin{definition}\label{pathwhitenoise}
	We define the pathwise $W$-gaussian white noise as the functional 
    $\dot{B}_W \in H^{-1}_{W,V,\mathcal{D}}(\mathbb{T}):=(H_{W,V,\mathcal{D}}(\mathbb{T}))^\ast$ given by
	\begin{equation}\label{bponto}
		\dot{B}_W(g) = -\int_{\mathbb{T}} B_W(s-) D_W^-g(s) dW(s).
	\end{equation}
\end{definition}

\begin{remark}
	Observe that the expression in the right-hand side of \eqref{bponto} is, almost surely, well-defined in the pathwise sense, as it is an integral with respect to the measure induced by $W$, and $B_W(\cdot)$, almost surely, has c\`adl\`ag sample paths.
\end{remark}

We will now study the existence and uniqueness of weak solutions of the following stochastic partial differential equation:
\begin{equation}\label{elliptic}
	\kappa^2 u - D_V^+ (H D_W^- u) = \dot{B}_W
\end{equation}
on the space $H_{W,V,\mathcal{D}}(\mathbb{T})$.

It is important to notice that equation \eqref{elliptic} can be seen as $W$-generalized counterpart of a special non-fractional case of the well-known Whittle--Matérn stochastic partial differential equation, e.g., \cite{lindgren2011explicit}.

\begin{remark}
    The fact that we can construct $\dot{B}_W$ relative to $L_V^2(\mathbb{T})$ for any strictly increasing function $V:\mathbb{R}\to\mathbb{R}$ satisfying \eqref{periodic} is essential due to the fact that $V$ is completely determined by the original problem in which the differential operator $\kappa^2 I - D_V^+ H D_W^-$ was modeled upon. See, for instance, \cite[Theorem 15.7]{dynkin1965markov} as an example in which some particular cases of $V$ and $W$ arise naturally.
\end{remark}

Let us now, provide the definition of weak solution of the above equations:

\begin{definition}\label{weaksolmatern}
	We say that $u\in H_{W,V,\mathcal{D}}(\mathbb{T})$ is a weak solution of equation \eqref{elliptic} if for every $g\in H_{W,V,\mathcal{D}}(\mathbb{T})$, the following identity is true:
	$$\int_{\mathbb{T}} \kappa^2 u g dV + \int_{\mathbb{T}} HD_W^- u D_W^-g dW = \int_{[0,1)} gdB_W,$$
	or, equivalently, in terms of functionals, if the following identity is true:
	$$B_{L_{W,V}}(u, g) = \dot{B}_W(g),$$
	where $B_{L_{W,V}}$ is the bilinear functional given by $$\int_{\mathbb{T}} \kappa^2 u g dV + \int_{\mathbb{T}} HD_W^- u D_W^-g dW$$
	for $u,g\in H_{W,V,\mathcal{D}}(\mathbb{T})$ and $\dot{B}_W$ is the pathwise $W$-Gaussian white noise defined in (\ref{bponto}).
\end{definition}

By comparing with the solutions of the standard Whittle--Mat\'ern equation, we can readily see that equation \eqref{pathwhitenoise} can be suitable for modelling situations in which the process finds barriers, as well as, having eventually non-diffusive behavior throughout the domain.

We, then, have the following theorem on existence and uniqueness of solutions of \eqref{elliptic}:

\begin{theorem}\label{exisuniqellipt}
	Let $\kappa$ and $H$ be bounded functions that are also bounded away from zero, where $H$ is positive. Then, equation \eqref{elliptic} has, almost surely, a unique solution with c\`adl\`ag sample paths.
\end{theorem}
\begin{proof}
	From \eqref{functionalwhite}, $\dot{B}_W$ is, almost surely, an element of $H^{-1}_{W,V,\mathcal{D}}(\mathbb{T})$. Therefore, from Lax-Milgram theorem, \eqref{elliptic} has, almost surely, a unique solution $u\in H_{W,V,\mathcal{D}}(\mathbb{T})$. Finally, from Theorem \ref{sobchar}, $u$ has, almost surely, c\`adl\`ag sample paths. 
\end{proof}

\section*{Declarations}

\begin{itemize}
	\item \textbf{Founding}: Research Supported by CAPES (grant No. 88882.440725/2019-01)
	\item \textbf{Availability of Technical Matherials} : Not applicable
	\item \textbf{Ethical approval}: Not applicable
\end{itemize}

\bibliography{sn-bibliography}

\end{document}